%% file: 3DquinticGP.tex
\documentclass[11pt,reqno]{amsart}
\usepackage{amsmath, amsfonts, amssymb, amsthm, amsxtra, array, stmaryrd, graphicx, hyperref}
\usepackage{textcomp, verbatim, color, mathrsfs, mathabx, braket}
\usepackage{epsfig, graphicx, caption}
\usepackage[margin=1in]{geometry}
\usepackage[nocompress]{cite}
\newcommand{\la}{\langle}
\newcommand{\ra}{\rangle}

\newcommand{\pphi}{|\phi\ra\la\phi|}

\newtheorem{thm}{Theorem}[section]
\newtheorem{lemma}{Lemma}[section]
\newtheorem{hyp}[lemma]{Hypothesis}
\newtheorem{prop}[lemma]{Proposition}

\numberwithin{equation}{section}
\allowdisplaybreaks
  
\title{Uniqueness of solutions to the 3D quintic Gross-Pitaevskii Hierarchy}

\author{ Younghun Hong }
\address{Department of Mathematics \newline\indent The University of Texas at Austin}
\email{yhong@math.utexas.edu}

\author{ Kenneth Taliaferro }
\address{Department of Mathematics \newline\indent The University of Texas at Austin}
\email{ktaliaferro@math.utexas.edu}

\author{ Zhihui Xie }
\address{Department of Mathematics, Statistics and Computer Science \newline\indent University of Illinois at Chicago}
\email{zxie@uic.edu}

\date{\today}

\begin{document}
  
\begin{abstract}

In this paper, we study solutions to the three-dimensional quintic Gross-Pitaevskii hierarchy. We prove unconditional uniqueness among all small solutions in the critical space $\mathfrak{H}^1$ (which corresponds to $H^1$ on the NLS level).  With slight modifications to the proof, we also prove unconditional uniqueness of solutions to the Hartree hierarchy without a smallness condition.  Our proof uses the quantum de Finetti theorem, and is an extension of the work by Chen-Hainzl-Pavlovi\'c-Seiringer \cite{CHPS}, and our previous work \cite{UniqueLowReg}.

\end{abstract}

\maketitle

\section{Introduction}   \label{sec: introduction}

\subsection{Statement of the main result}

In this paper, we establish uniqueness of small solutions to the three-dimensional quintic Gross-Pitaevskii (GP) hierarchy in the scaling-critical Sobolev type space.

The 3d quintic GP hierarchy is an infinite system of coupled linear equations 
\begin{equation}   \label{hierarchy eqn}
 i\partial_t \gamma^{(k)} = (-\Delta_{\underline{x}_k}+\Delta_{\underline{x}_k'}) \gamma^{(k)} + \lambda \sum_{j=1}^k B_{j;k+1,k+2}\gamma^{(k+2)},   \quad k\in\mathbb{N},
\end{equation}
where $\gamma^{(k)}=\gamma^{(k)}(t, \underline{x}_k;\underline{x}_k'): [0,T)\times\mathbb{R}^{3k}\times\mathbb{R}^{3k}\to\mathbb{C}$, the underlined variables $\underline{x}_k$ and $\underline{x}_k'$ denote $k$-tuples of spacial variables, i.e., $\underline{x}_k=(x_1,x_2,\cdots,x_k)\in\mathbb{R}^{3k}$ and $\underline{x}_k'=(x_1',x_2',\cdots,x_k')\in\mathbb{R}^{3k}$, and the Laplacians are given by $\Delta_{\underline{x}_k}:=\sum_{j=1}^k \Delta_{x_j}$ and $\Delta_{\underline{x}_k'}:=\sum_{j=1}^k \Delta_{x_j'}$. We assume that for each $k\in\mathbb{N}$, $\gamma^{(k)}$ is a symmetric marginal density matrix such that 
\begin{equation}  \label{hermitian}
  \gamma^{(k)}(t,\underline{x}_k;\underline{x}_k')=\overline{\gamma^{(k)}(t,\underline{x}_k';\underline{x}_k)}
\end{equation}
and 
\begin{equation}  \label{sym}
\gamma^{(k)}(t,x_{\sigma(1)},\cdots,x_{\sigma(k)};x_{\sigma'(1)}',\cdots,x_{\sigma'(k)}')=\gamma^{(k)}(t,\underline{x}_k; \underline{x}_k') 
\end{equation}
for any permutations $\sigma$ and $\sigma'$ on $\{1, 2, \cdots,k\}$. The \emph{contraction operator} $B_{j;k+1,k+2}$ is defined by
\begin{equation}   \label{Bj operator GP}
 \begin{split}
 & B_{j;k+1,k+2} \gamma^{(k+2)} (t,\underline{x}_k;\underline{x}_k') \\
 :&= \int dx_{k+1}dx_{k+2} dx_{k+1}'dx_{k+2}' [\delta(x_j-x_{k+1})\delta(x_j-x_{k+2})\delta(x_j-x_{k+1}')\delta(x_j-x_{k+2}')  \\
 &\quad\quad\quad\quad\quad\quad\quad - \delta(x_j'-x_{k+1})\delta(x_j'-x_{k+2})\delta(x_j'-x_{k+1}')\delta(x_j'-x_{k+2}') ]
 \gamma^{(k+2)}(t,\underline{x}_{k+2};\underline{x}_{k+2}')  \\
 & = \gamma^{(k+2)}(t,\underline{x}_{k},x_j,x_j;\underline{x}_{k}',x_j,x_j) - \gamma^{(k+2)}(t,\underline{x}_{k},x_j',x_j';\underline{x}_{k}',x_j',x_j').
\end{split}
\end{equation}  
The coupling constant is either $-1$ or $1$. We call the GP hierarchy \eqref{hierarchy eqn} \emph{defocusing} if $\lambda=1$, and \emph{focusing} if $\lambda=-1$.

To define solutions to the GP hierarchy, we introduce the following definitions (see also \cite{ESY06, ESY07, ESY09, ESY10, CHPS}). For $s\geq 0$, we define the homogeneous Sobolev space $\dot{\mathfrak{H}}^s$ for sequences by 
\begin{equation}  \label{dotfrakH}
\dot{\mathfrak{H}}^s:=\Big\{\{\gamma^{(k)}\}_{k\in\mathbb{N}}: \text{Tr }(|R^{(k,s)}\gamma^{(k)}|) < M^{2k} \text{ for some positive constant $M<\infty$}\Big\}
\end{equation}
where
$$R^{(k,s)} :=\prod_{j=1}^k (-\Delta_{x_j})^{\frac{s}{2}}(-\Delta_{x'_j})^{\frac{s}{2}}.$$
Similarly, we define the inhomogeneous Sobolev space $\mathfrak{H}^s$ for sequences by 
\begin{equation} \label{frakH}
 \mathfrak{H}^s:=\Big\{\{\gamma^{(k)}\}_{k\in\mathbb{N}}: \text{Tr }(|S^{(k,s)}\gamma^{(k)}|) < M^{2k} \text{ for some constant $M<\infty$}\Big\}
\end{equation}
where
$$S^{(k,s)} :=\prod_{j=1}^k (1-\Delta_{x_j})^{\frac{s}{2}}(1-\Delta_{x'_j})^{\frac{s}{2}}.$$
A sequence $\{\gamma^{(k)}(t)\}_{k\in\mathbb{N}}$ is called a \emph{mild solution} in $L_{t\in[0,T)}^\infty \dot{\mathfrak{H}}^s$ (or $L_{t\in[0,T)}^\infty \mathfrak{H}^s$) to the quintic GP hierarchy if it solves the hierarchy of the integral equations
\begin{equation}  \label{integral eqn}
\gamma^{(k)}(t)=U^{(k)}(t)\gamma^{(k)}(0)+i\lambda\sum_{j=1}^k \int_0^t U^{(k)}(t-s)B_{j;k+1,k+2}\gamma^{(k+2)}(s)ds,\quad \forall k\in\mathbb{N},
\end{equation}
where $U^{(k)}(t):=e^{it(\Delta_{\underline{x}_k}-\Delta_{\underline{x}'_k})}$ is the free evolution operator. A sequence $\{\gamma^{(k)}\}_{k\in\mathbb{N}}$ is called \emph{admissible} if for each $k\in\mathbb{N}$ and $t\in [0,T)$, $\mathcal{\gamma}^{(k)}$ is a non-negative trace class operator on $L_{sym}^2(\mathbb{R}^{3k}\times\mathbb{R}^{3k})$ (subset of $L^2$ functions that satisfy \eqref{sym}) and 
\begin{equation}  \label{admissibility}
 \gamma^{(k)}=\textup{Tr}_{k+1}(\gamma^{(k+1)})=\int_{\mathbb{R}^3}dx_{k+1}\gamma^{(k+1)}(\underline{x}_k, x_{k+1};\underline{x}_k', x_{k+1}).
\end{equation}
We call a sequence $\{\gamma^{(k)}\}_{k\in\mathbb{N}}$ a \emph{limiting hierarchy} if there is a sequence $\{\gamma_N^{(N)}\}_{N\in\mathbb{N}}$ of non-negative density matrices on $L_{sym}^2(\mathbb{R}^{3N}\times \mathbb{R}^{3N})$ with $\textup{Tr}(\gamma_N^{(N)})=1$ such that $\gamma^{(k)}$ is the weak-* limit of the $k$-particle marginals of $\gamma_N^{(N)}$ in the trace class on $L_{sym}^2(\mathbb{R}^{3k}\times\mathbb{R}^{3k})$, that is,
\begin{equation}
\begin{aligned}
\gamma_N^{(k)}:&=\textup{Tr}_{k+1,\cdot\cdot\cdot, N}(\gamma_N^{(N)})\\
&=\int_{\mathbb{R}^{3(N-k)}}dx_{k+1}\cdots dx_{N}\gamma_N^{(N)}(\underline{x}_k, x_{k+1},\cdots x_N;\underline{x}_k', x_{k+1},\cdots, x_N)\\
&\rightharpoonup^*\gamma^{(k)}\textup{ as }N\to\infty.
\end{aligned}
\end{equation}

In this paper, we consider mild solutions to the GP hierarchy \eqref{hierarchy eqn} that are admissible or limiting hierarchies. Such mild solutions are physically relevant in the theory of derivation of the nonlinear Schr\"odinger equation (NLS) from the many body linear Schr\"odinger equation (see Section \ref{subsec: related work}).

We now state our main result.

\begin{thm}[Uniqueness of small solutions to the quintic GP hierarchy]  \label{thm: main thm GP}
Suppose that $\{\gamma^{(k)}(t)\}_{k\in\mathbb{N}}$ is a mild solution in $L_{t\in[0,T)}^\infty\dot{\mathfrak{H}}^1$ to the quintic GP hierarchy \eqref{hierarchy eqn} with initial data $\{\gamma^{(k)}(0)\}_{k\in\mathbb{N}}$, which is either admissible or a limiting hierarchy for each $t$. If $\text{Tr }(|R^{(k,1)}\gamma^{(k)}|) < M^{2k}$ for all $t\in[0,T)$ for $M>0$ sufficiently small, then $\{\gamma^{(k)}(t)\}_{k\in\mathbb{N}}$ is the only such solution for the given initial data. 
\end{thm}

The quintic GP hierarchy is closely related to the quintic NLS via factorized functions. Indeed, one can check that if $\phi_t$ is a solution to the quintic NLS
\begin{equation}   \label{quintic nls}
 i\partial_t \phi_t = (-\Delta)\phi_t +\lambda |\phi_t|^4\phi_t,
\end{equation}
then a sequence of factorized functions,
\begin{equation}    \label{fac sol}
 \gamma^{(k)}(t,\underline{x}_k;\underline{x}_k') = (\Ket{\phi_t} \Bra{\phi_t})^{\otimes{k}} :=\prod_{j=1}^k \phi_t(x_j)\overline{\phi_t(x_j')},
\end{equation}
solves the GP hierarchy \eqref{hierarchy eqn}. In this sense, proving uniqueness for the GP hierarchy is more difficult than it is for the quintic NLS.

The quintic GP hierarchy was studied by T. Chen and Pavlovi\'c \cite{CPquintic} for the derivation of the quintic NLS as the Gross-Pitaevskii field limit of a non-relativistic Bose gas with $3$-particle interactions. As a part of their analysis, the authors proved (conditional) uniqueness of solutions to the quintic GP hierarchy in an energy space, that is, a Sobolev type space of order 1, in one and two dimensions. We remark that  in all dimensions, proving such uniqueness in an energy space  is necessary to derive NLS. However, it is an open problem to prove uniqueness in three dimensions.

Theorem \ref{thm: main thm GP} provides an answer for this open problem under a smallness assumption. We remark that the 3d quintic GP hierarchy is scaling-critical in $\dot{\mathfrak{H}}^1$, and that even with our smallness assumption, our theorem is the first uniqueness theorem for the cubic or quintic GP hierarchy in a scaling-critical space. Moreover, uniqueness in Theorem \ref{thm: main thm GP} is unconditional.

It remains an open problem to remove the smallness assumption.  In the case of the 3d quintic NLS, it is known that solutions are unique in the space $H^s$ for $s\geq 1$, without a smallness assumption \cite{kato95, HanFang, Cazenave, I-team}.  However, the proof of unconditional uniqueness in the scaling-critical case $s=1$ differs from the proof in the subcritical case $s>1$.  In the case of the 3d quintic GP hierarchy, we also expect that an approach different from the one that we use in the scaling-subcritical case is needed to remove the smallness assumption in the scaling-critical case.  Currently, the main obstacle to removing the smallness assumption for solutions to the 3d quintic GP hierarchy in the scaling-critical case is the generally infinite cardinality of the support of the measure $\mu$ in the statement of the quantum de Finetti theorem, Theorem \ref{thm: strong de Finetti}.

To compare scaling-critical and subcritical regimes, we provide a uniqueness theorem for the 3d quintic Hartree hierarchy. The 3d quintic Hartree hierarchy is also an infinite hierarchy as \eqref{hierarchy eqn}.  However the contraction operator $B_{j,k+1, k+2}$ in \eqref{Bj operator GP} is replaced by
\begin{align}  
 & B_{j;k+1,k+2} \gamma^{(k+2)} (t,\underline{x}_k;\underline{x}_k') \nonumber\\
 &:= \int dx_{k+1}dx_{k+2} dx_{k+1}'dx_{k+2}'\nonumber\\
 &\hspace{1cm} V(x_j-x_{k+1}, x_j-x_{k+2}) V(x_j-x_{k+1}', x_j-x_{k+2}')\gamma^{(k+2)}(t,\underline{x}_{k+2};\underline{x}_{k+2}') \label{Bj operator hartree}\\
 &\hspace{1cm}- \int dx_{k+1}dx_{k+2} dx_{k+1}'dx_{k+2}'\nonumber\\
 &\hspace{2cm} V(x_j'-x_{k+1}, x_j'-x_{k+2}) V(x_j'-x_{k+1}', x_j'-x_{k+2}')\gamma^{(k+2)}(t,\underline{x}_{k+2};\underline{x}_{k+2}'). \nonumber
\end{align}

Note that the 3d quintic Hartree equation is subcritical in $L_{t\in[0,T)}^\infty\mathfrak{H}^1$ if the three-particle interaction potential $V$ is less singular than the product of delta functions.  This is, if $V(\cdot,\cdot)\in L_{x,y}^r(\mathbb{R}^3\times\mathbb{R}^3)$ for some $r>1$. In this case, we can show unconditional uniqueness for the 3d quintic Hartree hierarchy without a smallness assumption.
\begin{thm}[Unconditional uniqueness for the quintic Hartree hierarchy]  \label{thm: main thm hartree}
Suppose that $V(\cdot,\cdot)\in L_{x,y}^r(\mathbb{R}^3\times\mathbb{R}^3)$ for some $r>1$. Let $\{\gamma^{(k)}(t)\}_{k\in\mathbb{N}} \in \dot{\mathfrak{H}}^1$ be a mild solution to the quintic Hartree hierarchy \eqref{integral eqn} with initial data $\{\gamma^{(k)}(0)\}_{k\in\mathbb{N}}$, which is either admissible or a limiting hierarchy for each $t$. If there exists $M>0$ such that $\text{Tr }(|R^{(k,1)}\gamma^{(k)}|) < M^{2k}$ for all $t\in[0,T)$, then  $\{\gamma^{(k)}(t)\}_{k\in\mathbb{N}}$ is the only such solution for the given initial data. 
\end{thm}

\subsection{Related works}   \label{subsec: related work}

The background work in this line goes back to the derivation of Schr\"odinger type equations from interacting particle systems. In the pioneering works by Hepp \cite{Hepp}, Spohn \cite{Spohn} and in a series of more recent breakthroughs by Erd\"os, Schlein and Yau \cite{ESY06, ESY07, ESY09, ESY10}, the authors derived the cubic NLS in $\mathbb{R}^3$. A major ingredient in this derivation is the establishment the uniqueness of solutions to the corresponding GP hierarchy. The proof of uniqueness by Erd\"os-Schlein-Yau requires sophisticated Feynman graph expansions. Later, Klainerman and Machedon \cite{KM} rephrased this as a board game argument to provide an alternative approach to prove uniqueness of solutions. However, the result in \cite{KM} is conditional in that the solutions that satisfy an a-priori space-time bound assumption. This assumption is used by Kirkpatrick, Schlein, and Staffilani \cite{KSS} in two dimensional settings for compact and non-compact domains. \\

A recent new proof on the unconditional uniqueness of 3d cubic GP hierarchy was initiated by T.Chen, Hainzl, Pavlovi\'c and Seringer \cite{CHPS} using the \emph{quantum de Finetti theorem}. The quantum de Finetti theorem is a quantum analogue of the Hewitt-Savage theorem in probability theory. The strong version of the quantum de Finetti theorem (see \ref{thm: strong de Finetti}) asserts that an infinite sequence of \emph{admissible} marginal density matrices can be expressed as an average over factorized states. However, for each $t$, the limiting hierarchies of density matrices do not necessarily satisfy admissibility. In this case, one uses the weak version of the de Finetti theorem (see \ref{thm: weak de Finetti}).  This is necessary when working with the BBGKY hierarchy approach for the derivation of NLS as in \cite{ESY06,ESY07,ESY09,ESY10}, where one starts with a finite BBGKY hierarchy of $N$ equations for the bosonic $N$-particle system (see (2.1) in \cite{ESY07}). In this case, the GP hierarchy of 
equations is obtained by taking $N\to \infty$ in the finite hierarchy. As part of the derivation, one proves that the weak-$*$ limit of solutions $\gamma_N^{(k)}$ 
to the BBGKY hierarchy solve the infinite GP hierarchy. 

By taking advantage of the quantum de Finetti theorems that give an alternative factorized formula for the solutions to the hierarchy, the authors of \cite{UniqueLowReg} established unconditional uniqueness for cubic GP hierarchy at the same regularity level of the corresponding NLS. Others have also used the de Finetti theorem to prove unconditional uniqueness for GP hierarchies in various settings. In \cite{verdanCubicT3}, V. Sohinger adapted the method from \cite{CHPS} to cubic GP hierarchy in a periodic setting. In \cite{ChenSmith}, X.Chen-Smith studied a Chen-Simon-Schr\"odigner hierarchy.

\subsection{Strategy of the proof}  \label{subsec: strategy}

We prove Theorem \ref{thm: main thm GP} and Theorem \ref{thm: main thm hartree} in the framework of Chen-Hainzl-Pavlovi\'c-Seringer \cite{CHPS}. Due to the linearity of the hierarchy, it suffices to show that solutions solution having a zero initial are the zero solution. In our proof, we iterate the Duhamel formula \eqref{integral eqn} with zero initial data $n$ times, resulting in a number of terms that grows factorially in $n$. We reduce the number of terms by the Erd\"os-Schlein-Yau combinatorial argument in Klainerman-Machedon's formulation \cite{KM}. The quintic version of this combinatoric reduction was used by Chen-Pavlovic in \cite{CPquintic}. We use it for the 3d quintic GP and Hartee hierarchies without modification. Next, we apply the quantum de Finetti theorem to write each term as an integral sum of factorized states, and reorganize them using a tree-graph structure (see Figure 1 below) which extends the tree-graph in Chen-Hainzl-Pavlovi\'c-Seiringer \cite{CHPS}. Then, we iteratively estimate 
the $n$ integrals.  In each step, we apply our multilinear estimates, which can be found in Appendix \ref{sec: multilin estimates}.  Finally, we send $n\to\infty$ and find that solutions having zero initial data must be the zero solution.

In our previous work \cite{UniqueLowReg}, we proved unconditional uniqueness for the cubic GP hierarchy in a low regularity setting, using a similar approach. In \cite{UniqueLowReg}, our key ingredients were the trilinear estimates $(2.19)$, $(2.21)$ and $(2.23)$ in Lemma 2.6.  These estimates are based on the dispersive estimates
\begin{equation}\label{dispersive estimate}
\|e^{it\Delta}f\|_{L^{p}(\mathbb{R}^d)}\lesssim|t|^{-d(\frac{1}{2}-\frac{1}{p})}\|f\|_{L^{p'}(\mathbb{R}^d)},\quad p\geq 2,
\end{equation}
and negative order Sobolev norm estimates (Lemma A.3 in \cite{UniqueLowReg}). In the proof, we applied these estimates to the reorganized integrals iteratively together with multilinear estimates based on Strichartz estimates ($(2.20)$, $(2.22)$ and $(2.24)$ in Lemma 2.6). We remark that the use of dispersive estimates is crucial in obtaining the optimal subcritical low regularity uniqueness theorem.  The dispersive estimates don't work in the scaling-critical space, however. Roughly speaking, this is due to the failure of integrability (in time) of the bound in \eqref{dispersive estimate}. For instance, if one tries to prove uniqueness for the 3d quintic GP hierarchy in $L_{t\in[0,T)}^\infty \mathfrak{H}^1$ by the same approach, one should choose $p=6$ for the multilinear estimate. Then, the bound in \eqref{dispersive estimate} is not integrable in time.

In the present work, instead of using dispersive estimates, we use multilinear estimates (Proposition \ref{prop:multilinear gp} and Propositions \ref{prop:multilinear hartree}) that are based on by Strichartz estimates and a negative order Sobolev norm bound. In the case of the Hartree hierarchy, we also make use of a convolution estimates of W. Beckner \cite{beckner}.

\subsection{Notation}
In order to prove Theorem \ref{thm: main thm GP} and Theorem \ref{thm: main thm hartree} at the same time, we define
\begin{align}
V_\infty(y,z):=
\begin{cases}
V(y,z), &\text{ for the Hatree hierarchy}. \\ 
\lambda\,\delta(y)\delta(z), &\text{ for the GP hierarchy}. 
\end{cases} \label{potential cases}
\end{align}
With this notation, we can now combine definitions \eqref{Bj operator GP} and \eqref{Bj operator hartree} of $B_{j;k+1,k+2}$ for the GP hierarchy and the Hartree hierarchy, respectively, as follows.
\begin{align}  
 & B_{j;k+1,k+2} \gamma^{(k+2)} (t,\underline{x}_k;\underline{x}_k') \nonumber\\
 &:= \int dx_{k+1}dx_{k+2} dx_{k+1}'dx_{k+2}'\nonumber\\
 &\hspace{1cm} V_\infty(x_j-x_{k+1}, x_j-x_{k+2}) V_\infty(x_j-x_{k+1}', x_j-x_{k+2}')\gamma^{(k+2)}(t,\underline{x}_{k+2};\underline{x}_{k+2}') \label{Bj operator general}\\
 &\hspace{1cm}- \int dx_{k+1}dx_{k+2} dx_{k+1}'dx_{k+2}'\nonumber\\
 &\hspace{2cm} V_\infty(x_j'-x_{k+1}, x_j'-x_{k+2}) V_\infty(x_j'-x_{k+1}', x_j'-x_{k+2}')\gamma^{(k+2)}(t,\underline{x}_{k+2};\underline{x}_{k+2}'). \nonumber
\end{align}


\subsection{Organization of the paper}
This paper is organized as follows.  In section \ref{sec: comb} we present the road map for the proof of the main theorems and reduce the the main theorems to Proposition \ref{prop: zero trace norm}. We illustrate with an example how to factorize solutions in section \ref{sec: example factorization}.  In section \ref{sec: distinguished tree graph}, we introduce tree graphs to illustrate our decomposition of each factor, and present properties of the associated kernels. The proof of Proposition \ref{prop: zero trace norm} occupies section \ref{sec: recursive bounds}. In appendix \ref{sec: multilin estimates}, we prove several multilinear estimates that we use section \ref{sec: recursive bounds}.

\section{Outline of the Proof}   \label{sec: comb}

We describe the strategy to prove uniqueness in more detail.

\subsection{Setup}
Let $\{\gamma_1^{(k)}(t)\}_{k\in\mathbb{N}}$ and $\{\gamma_2^{(k)}(t)\}_{k\in\mathbb{N}}$ be two mild solutions in $L_{t\in[0,T)}^{\infty}\dot{\mathfrak{H}}^1$ that solve \eqref{integral eqn} with the same initial data, and are either admissible or limiting hierarchies. To prove uniqueness, we will show that their difference $\{\gamma^{(k)}(t)\}_{k\in\mathbb{N}}$, given by
\begin{equation}   \label{diff mild sols}
 \gamma^{(k)}(t):=\gamma_1^{(k)}(t)-\gamma_2^{(k)}(t),  \qquad k\in\mathbb{N},
\end{equation}
is zero. By linearity, the difference $\{\gamma^{(k)}(t)\}_{k\in\mathbb{N}}$ solves the GP (or Hartree) hierarchy with zero initial data. Therefore, it suffices to prove the following. 
\begin{prop}  \label{prop: zero trace norm}
Suppose that $\{\gamma^{(k)}(t)\}_{k\in\mathbb{N}}$ is a mild solution to \eqref{hierarchy eqn} with zero initial data, and that it is either admissible or a limiting hierarchy.\\
$(i)$ If $\{\gamma^{(k)}(t)\}_{k\in\mathbb{N}}$ solves the quintic GP hierarchy and $\|\{\gamma^{(k)}(t)\}_{k\in\mathbb{N}}\|_{L_{t\in[0,T)}^\infty\dot{\mathfrak{H}}^1}$ is sufficiently small, then
\begin{equation}  \label{zero trace norm}
{\rm Tr}(|R^{(k,-1)}\gamma^{(k)}(t)|)=0, \quad \forall k\in\mathbb{N}. 
\end{equation}
$(ii)$ If $\{\gamma^{(k)}(t)\}_{k\in\mathbb{N}}$ solves the quintic Hartree hierarchy and $V\in L^{1+}$, then \eqref{zero trace norm} holds.
\end{prop}

\subsection{Duhamel expansion}
To show \eqref{zero trace norm}, we first generate a Duhamel expansion as follows. For each $k\in\mathbb{N}$, $\gamma^{(k)}(t)$ solves 
\begin{equation}  \label{diff integral eqn}
 \gamma^{(k)}(t)=i\lambda \sum_{j=1}^k \int_0^t U^{(k)}(t-t_1)B_{j;k+1,k+2}\gamma^{(k+2)}(t_1)dt_1.
\end{equation}
Fix $k\in\mathbb{N}$. Iterating the integral equation \eqref{diff integral eqn} $(n-1)$ times, we write 
\begin{equation}  \label{Duhamel n fold}
 \gamma^{(k)}(t)=(i\lambda)^n \int_{t_{n}\leq \cdots \leq t_1\leq t} U^{(k)}(t-t_1)B_{k+2} \cdots U^{(k+2n-2)}(t_{n-1}-t_{n})B_{k+2n}\gamma^{(k+2n)}(t_{n})dt_1\cdots dt_{n}.
\end{equation}
Here, for each $r\geq 1$, the \emph{combined contraction operator} is the sum of $k+2(r-1)$ many operators,
$$B_{k+2r}:=\sum_{j=1}^{k+2(r-1)} B_{j;k+2r-1,k+2r}.$$ 
For notational convenience, we introduce the following notation. 
$$U_{j,j'}^{(i)}:=U^{(i)}(t_j-t_{j'}),$$
$$\underline{t}_{n} := (t,t_1,\cdots,t_{n}),\quad t_0=t,$$
$$J^{k}(\underline{t}_{n}):=U_{0,1}^{(k)}B_{k+2}U_{1,2}^{(k+2)}B_{k+4} \cdots U_{n-1,n}^{(k+2n-2)}B_{k+2n}\gamma^{(k+2n)}(t_{n}).$$
Then $\gamma^{(k)}(t)$ in \eqref{Duhamel n fold} can be expressed in a compact form as
\begin{equation}   \label{Duhamel n fold in Jk}
 \gamma^{(k)}(t)= (i\lambda)^n\int_{t_{n}\leq \cdots \leq t_1\leq t} J^k(\underline{t}_{n})d\underline{t}_{n}.
\end{equation}
 
One may have observed that for fixed $k$, the number of terms in $J^{k}(\underline{t}_{n})$ is $k(k+2)\cdots(k+2n-2)\sim \mathcal{O}((2n)!)$. This factorial growth on the number of Duhamel expansion terms is the first difficulty before we proceed with the proof of proposition \ref{prop: zero trace norm}. As a preparation, we will present a summary of the combinatorial reduction process in section \ref{subsec: comb reduction} to reduce $J^{k}(\underline{t}_{n})$ into a smaller number of terms that we can control.

\subsection{Combinatorial reduction} \label{subsec: comb reduction}
 
In the celebrated works \cite{ESY06, ESY07, ESY09, ESY10}, Erd\"os-Schlein-Yau developed a sophisticated combinatorial arguments to reduce the number of Duhamel terms. Later, Klainerman and Machedon \cite{KM} rephrased this as a board game, which was extended to the quintic GP hierarchy by Chen-Pavlovi\'c in \cite{CPquintic}. Since we will use the same arguments, we only present the notation and key reduction steps in this section. We refer the readers to \cite{CPquintic} for the proofs of the related lemmas and theorems.

Let $\sigma$ be a map from $\{k + 1, k + 2, \cdots, k + 2n - 1 \}$ to $\{1, 2, 3, \cdots, k + 2n - 2 \}$ such that $\sigma(2)=1$ and $\sigma(j)<j$ for all $j$. $\mathcal{M}_{k,n}$ denotes the set of all such mappings. Then we have that 
\begin{equation}  \label{Jk and Jksigma}
 J^{k}(\underline{t}_{n}) = \sum_{\sigma \in \mathcal{M}_{k,n}} J^{k}(\underline{t}_{n};\sigma), 
\end{equation}
where
\begin{equation}   \label{Jk before factorization}
 J^k(\underline{t}_{n};\sigma)=U^{(k)}_{0,1}B_{\sigma(k+1);k+1,k+2} U^{(k+2)}_{1,2}\cdots U^{(k+2n-2)}_{n-1,n}B_{\sigma(k+2n-1);k+2n-1,k+2n}(\gamma^{(k+2n)}(t_{n}))
\end{equation}
is a basic term in $J^{k}(\underline{t}_{n})$.
 
Next, for each $\sigma\in \mathcal{M}_{k,n}$ there is a $(k+2n-1)\times n$ matrix corresponding to it. This matrix can be reduced to a special upper echelon matrix that corresponds to $\sigma_s$ via finite many so called \emph{acceptable moves}. This transformation defines an equivalence relation among all the maps in $\mathcal{M}_{k,n}$. If $\sigma$ and $\sigma_s$ are equivalent, we denote this equivalence by $\sigma \sim \sigma_s$. From each equivalence classes, we pick one map that corresponds to a special upper echelon matrix, denote it by $\sigma_s$. Theorem 7.4 in \cite{CPquintic} confirms that there is a subset $D_{\sigma_s,t} \subset [0,t]^n$, such that
\begin{equation}   \label{equiv class integral}
\sum\limits_{\sigma \sim \sigma_s}\int_0^{t} ... \int_0^{t_{n-1}} J^k(\underline{t}_{n};\sigma)dt_{1} \dots dt_{n} = \int\limits_{D_{\sigma_s, t}} J^k(\underline{t}_{n};\sigma_s)dt_{1} \dots dt_{n}.
\end{equation}
Hence we have a new formula for $\gamma^{(k)}(t)$
\begin{equation}  \label{Duhamel n fold in Jk sigma}
  \gamma^{(k)}(t)=\sum_{\sigma \in \mathcal{M}^s_{k,n}} \int_{D_{\sigma,t}} J^k(\underline{t}_{n};\sigma) d\underline{t}_{n},
\end{equation} 
where $\mathcal{M}^s_{k,n}$ is the union of all maps that correspond to special upper echelon matrices. By Lemma 7.3 of \cite{CPquintic}, $\#(\mathcal{M}^s_{k,n}) \leq 2^{k+3n-2}.$ \footnote[3]{The multiplier $2^{k+3n-2}$ is affordable to us, since it can be absorbed in $(CT)^{n}$.} 

\subsection{Quantum de Finetti theorem}
After decomposing $\gamma^{(k)}$ into a sum, we use the \emph{quantum de Finetti} theorems to express each term in a factorized form. The quantum de Finetti theorem has a strong and weak version, and pertains to to bosonic density matrices that are either admissible or obtained as a weak-$*$ limit, respectively. We state both the strong and weak versions \cite{LNR} below to be used in section \ref{subsec: comb reduction}.

\begin{thm}[Strong quantum de Finetti theorem] \label{thm: strong de Finetti}
If a sequence $\{\gamma^{(k)}\}_{k\in\mathbb{N}}$ of bosonic density matrices on $L_{sym}^2(\mathbb{R}^{3k})$ is admissible, then there exists a unique Borel probability measure $\mu$, supported on the unit sphere $S\subset L^2(\mathbb{R}^3)$ and invariant under multiplication of $\phi\in L^2(\mathbb{R}^3)$ by complex numbers of modulus one, such that
\begin{equation}  \label{de finetti cond}
\gamma^{(k)}=\int d\mu(\phi) (|\phi\ra\la\phi|)^{\otimes k},\quad k\in \mathbb{N}.  
\end{equation}
\end{thm}

\begin{thm}[Weak quantum de Finetti theorem]\label{thm: weak de Finetti}
If a sequence $\{\gamma^{(k)}\}_{k\in\mathbb{N}}$ of bosonic density matrices on $L_{sym}^2(\mathbb{R}^{3k})$ is a limiting hierarchy, then there exists a unique Borel probability measure $\mu$, supported on the unit ball $\mathcal{B}\subset L^2(\mathbb{R}^3)$ and invariant under multiplication of $\phi\in L^2(\mathbb{R}^3)$ by complex numbers of modulus one, such that \eqref{de finetti cond} holds.  
\end{thm}

There are different formulations of these theorems that are used in different settings. The formulation for density matrices was presented in a paper Lewin, Nam and Rougerie \cite{LNR}, and in a paper by Ammari and Nier \cite{ANphase}. For additional results related the de Finetti theorems, we refer the reader to Diaconis and Freedman \cite{DiaconisFreedman}, Hudson and Moody \cite{HudsonMoody}, and Stormer \cite{Stormer}.

To make sure the de Finetti theorems are applicable, we note that if $\{\gamma_1^{(k)}\}_k$ and $\{\gamma_2^{(k)}\}_k$ are admissible, then so is $\{\gamma^{(k)}\}_k$.  Similarly, if both $\{\gamma_1^{(k)}\}_k$ and $\{\gamma_2^{(k)}\}_k$ are obtained from a weak-$*$ limit, then so is $\{\gamma^{(k)}\}_k$. Thus by Theorem \ref{thm: strong de Finetti} and Theorem \ref{thm: weak de Finetti},  we obtain
\begin{equation}   \label{reduced deFinetti n fold}
  \gamma^{(k)}(t)=\sum_{\sigma \in \mathcal{M}^s_{k,n}} \int_{D_{\sigma,t}}d\underline{t}_{n} \int d\mu_{\underline{t}_{n}} J^k(\underline{t}_{n};\sigma).
\end{equation}
where 
\begin{equation}   \label{Jk in deFinetti}
 J^k(\underline{t}_{n};\sigma)=U^{(k)}_{0,1}B_{\sigma(k+1);k+1,k+2} U^{(k+2)}_{1,2}\cdots U^{(k+2n-2)}_{n-1,n}B_{\sigma(k+2n-1);k+2n-1,k+2n}(|\phi \rangle \langle\phi|)^{(k+2n)}.
\end{equation}

We remark that $J^k(\underline{t}_{n};\sigma)=J^k(\underline{t}_{n};\sigma;\underline{x}_k;\underline{x}'_k)$ depends on $\underline{x}_k,\underline{x}'_k$.   We omit the spatial variables for simplicity. We note that each factor in
\begin{equation*}
 (|\phi \rangle \langle\phi|)^{(k+2n)}(\underline{x}_{k+2n};\underline{x}'_{k+2n})=\prod_{i=1}^{k+2n}(|\phi \rangle \langle\phi|)(x_i;x'_i)
\end{equation*}
is a one-particle kernel, and that we can further decompose $J^k(\underline{t}_{n};\sigma)$ as
\begin{equation}  \label{Jk decomp}
 J^k(t,t_1,\cdots,t_{n};\sigma;\underline{x}_k;\underline{x}'_k)=\prod_{j=1}^k J_j^1(t,t_{l_j,1},\cdots,t_{l_j,m_j};\sigma_j;x_j;x'_j).
\end{equation}

To better explain the reduction procedure, we present an example in section \ref{sec: example factorization}, and then go back to the general case in section \ref{sec: distinguished tree graph}.

\section{Example Factorization}  \label{sec: example factorization}

Consider $k=2, n=4$, and $\rho$ a permutation of $\{1,2,\cdots,n\}$.  The map $\sigma_s$ is represented by the following upper echelon matrix (each highlighted entry in a row is to the left of each highlighted entry in a lower row)
\begin{equation}  \label{Jk example matrix}
 \begin{pmatrix}
 t_{\rho^{-1}(1)} & t_{\rho^{-1}(2)} & t_{\rho^{-1}(3)} & t_{\rho^{-1}(4)} \\
  \mathbf{B_{1;3,4}} & B_{1;5,6} & B_{1;7,8} & B_{1;9,10} \\
  B_{2;3,4} & \mathbf{B_{2;5,6}} & B_{2;7,8} & B_{2,9,10}  \\
  0 & B_{3;5,6} & B_{3;7,8} & B_{3;9,10} \\
  0 & B_{4;5,6} & \mathbf{B_{4;7,8}} & \mathbf{B_{4;9,10}} \\
  0 & 0 & B_{5;7,8} & B_{5;9,10}  \\
  0 & 0 & B_{6;7,8} & B_{6;9,10} \\
  0 & 0 & 0 & B_{7;9,10}  \\
  0 & 0 & 0 & B_{8;9,10}  \\
 \end{pmatrix}
\end{equation}
Then, we have
\begin{equation} \label{Jk example}
J^2(\underline{t}_{4};\sigma)=U^{(2)}_{0,1}B_{1;3,4} U^{(4)}_{1,2}B_{2;5,6} U^{(6)}_{2,3}B_{4;7,8} U^{(8)}_{3,4}B_{4;9,10}.
\end{equation}

We will organize the terms in expansion of $J^2(\underline{t}_{4};\sigma)$ into two one-particle density matrices by examining the effect of the contraction operators starting with the last one on the RHS of \eqref{Jk example}. We denote each factor in the last term $(\Ket{\phi} \Bra{\phi})^{\otimes 10}$ by $u_i$, ordered by increasing index $i$, so that $(\Ket{\phi} \Bra{\phi})^{\otimes 10}=\otimes_{i=1}^{10} u_i$.\\

First of all, in \eqref{Jk example}, the last interaction operator $B_{4;9,10}$ contracts the factor $u_4, u_9$ and $u_{10}$, and leaves all other factors unchanged.
\begin{equation}   \label{B4910 action}
B_{4;9,10} (\otimes_{i=1}^{10} u_i)=u_1\otimes u_2\otimes u_3\otimes \Theta_4 \otimes u_5\cdots \otimes u_8,
\end{equation}
where  $$\Theta_4:=B_{1;2,3}(u_4\otimes u_9\otimes u_{10}).$$
The index $\alpha$ in $\Theta_{\alpha}$ associates $\Theta_{\alpha}$ to the $\alpha$-th interaction operator from the left in \eqref{Jk example}. Since we only run the expansion to the $n$-th level, we have $1\leq \alpha\leq n$. In this specific case, $n=4$, and the $4$th interaction operator is $B_{4;9,10}$. \\

Next, $B_{4;7,8}$ contracts $U_{3,4}^{(8)}\Theta_4, U_{3,4}^{(8)}u_7$ and $U_{3,4}^{(8)}u_8$.
\begin{equation}   \label{B478 action}
 B_{4;7,8}U_{3,4}^{(8)}(\eqref{B4910 action})=(U_{3,4}^{(3)}(u_1\otimes u_2\otimes u_3)) \otimes \Theta_3 \otimes (U_{3,4}^{(2)}(u_5\otimes u_6)),
\end{equation}
where $$\Theta_3:=B_{1;2,3}((U_{3,4}^{(1)}\Theta_4)\otimes (U_{3,4}^{(1)}u_7)\otimes(U_{3,4}^{(1)}u_8)).$$ \\

Then, by the semigroup property, $U_{2,3}^{(i)} U_{3,4}^{(i)}=U_{2,4}^{(i)}$. The operator $B_{2;5,6}$ contracts $U_{2,4}^{(1)}u_2, U_{2,4}^{(1)}u_5$ and $U_{2,4}^{(1)}u_6$, which correspond to the 2nd, 5th, and 6th factors in \eqref{B478 action}.  The other factors are left invariant.
\begin{equation}   \label{B256 action}
 B_{2;5,6}U_{2,3}^{(6)} (\eqref{B478 action})=(U_{2,4}^{(1)}u_1)\otimes \Theta_2 \otimes(U_{2,4}^{(1)}u_3)\otimes (U_{2,3}^{(1)}\Theta_3),
\end{equation}
where $$\Theta_2=B_{1;2,3}(U_{2,4}^{(3)}(u_2\otimes u_5\otimes u_6)).$$ \\

Finally, $B_{1;3,4}$ contracts $U_{1,4}^{(1)}u_1$, $U_{1,4}^{(1)}u_3$, and $U_{1,3}^{(1)}\Theta_3$ and leaves other factors unchanged.
\begin{equation}   \label{B134 action}
 B_{1;3,4}U_{1,2}^{(4)} (\eqref{B256 action})=\Theta_1 \otimes (U_{1,2}^{(1)}\Theta_2),
\end{equation}
where $$\Theta_1=B_{1;2,3}((U_{1,4}^{(1)}u_1)\otimes (U_{1,4}^{(1)}u_3)\otimes (U_{1,3}^{(1)}\Theta_3)).$$
Therefore, $J^2$ can be factorized as
\begin{equation}   \label{factorized Jk example}
 J^2=(U_{0,1}^{(1)}\Theta_1) \otimes (U_{0,2}^{(1)}\Theta_2):=J^1_1 \otimes J^1_2.
\end{equation} \\

Now $J^2$ in \eqref{factorized Jk example} has two factors $J^1_j$ (note $j\leq k=2$), which are $1$-particle matrices. The reason we have such a decomposition is that $B_{\sigma_1(r);r,r+1}$ only affects three $u_i$ each time, and as the contraction processes, all the $u_i$ might be divided into different groups by the contraction connectivity. \\

For $j=1$, after replacing back $u_i=\Ket{\phi}\Bra{\phi}$, $i\leq k+2n=10$, we have
\begin{equation}   \label{relabel J11}
J^1_1=U_{0,1}^{(1)} B_{1;2,3} U_{1,3}^{(2)} B_{3;4,5} U_{3,4}^{(3)} B_{3;6,7}(\Ket{\phi}\Bra{\phi})^{\otimes 7}
\end{equation}
where we relabel the index in operators $B_{\sigma_1(r);r,r+1}$ such that the interaction operators in \eqref{relabel J11} correspond to $B_{1;3,4}, B_{4;7,8}, B_{4;9,10}$ respectively, and leave the connectivity structure among them unchanged. The labeling of function $\sigma_1$ (see the notation in \eqref{Jk decomp}) takes values $\sigma_1(2)=1, \sigma_1(4)=3$, and $\sigma_1(6)=3$. 

For $j=2$, we perform the relabeling in the same spirit find that
\begin{equation}   \label{relabel J12}
 J^1_2=U_{0,2}^{(1)} B_{1;2,3} U_{2,4}^{(3)} (\Ket{\phi}\Bra{\phi})^{\otimes 3},
\end{equation}
where $\sigma_2(2)=1$. \\

We note that for any $\ell<\ell'$, the interaction operators $B_{\sigma(\ell);\ell,\ell+1}$ and $B_{\sigma(\ell');\ell',\ell'+1} $ in $J^2$ (which are highlighted in \eqref{Jk example matrix}) belong to the same factor $J^1_j$ if either $\sigma(\ell)=\sigma(\ell')$ or $\sigma(\ell')=\ell$. In such cases, we consider them as being \emph{connected}. This connectivity structure is exactly the key point of the Duhamel terms that we want to illustrate using \emph{tree graphs}.  We include the detailed definitions and descriptions in section \ref{sec: distinguished tree graph}. \\

We further note that each $\sigma_j$ can be viewed as the restriction of $\sigma$ to $J^1_j$. We call factors that have a free propagator applied to each $\phi$ (like $J^1_2$) \emph{regular}, and factors that have the contractions of $(\Ket{\phi}\Bra{\phi})^{\otimes 3}$ without free propagator in between (like $J^1_1$) \emph{distinguished}. \\

\section{Tree graphs for the general case}   \label{sec: distinguished tree graph}


\subsection{The tree graphs}   \label{subsec: tree graph}
We begin by recalling that, from \eqref{Jk in deFinetti}, $J^k$ is given by
\begin{equation*} 
 J^k(\underline{t}_{n};\sigma)=U^{(k)}_{0,1}B_{\sigma(k+1);k+1,k+2} U^{(k+2)}_{1,2}\cdots U^{(k+2n-2)}_{n-1,n}B_{\sigma(k+2n-1);k+2n-1,k+2n}(|\phi \rangle \langle\phi|)^{\otimes (k+2n)}.
\end{equation*}
where
\begin{align*}
(\pphi)^{\otimes (k+2n)}(\underline{x}_{k+2n};\underline{x}_{k+2n}')=\prod_{i=1}^{k+2n}(\pphi)(x_i;x_i')
\end{align*}
is a product of one-particle kernels.  Since the free evolution operators $U_{j,j'}^{(i)}$ and the contraction operators $B_{\sigma(r);r,r+1}$ preserve the product structure, it follows that we can also decompose
\begin{align}
J^k(t,t_1,\dots,t_{n};\sigma;\underline{x}_{k};\underline{x}_{k}')=\prod_{j=1}^k J^1_j(t,t_{\ell_{j,1}},\dots,t_{\ell_{j,m_j}};\sigma_j;x_{j};x_{j}')\label{J decomposition}
\end{align}
into a product of one-particle kernels $J^1_j$.  We associate to this decomposition $k$ disjoint tree graphs $\tau_1, \tau_2, \dots, \tau_k$. These graphs appear as \emph{skeleton graphs} in \cite{ESY06, ESY07, ESY09, ESY10}. As in \cite{CHPS, UniqueLowReg}, we assign \emph{root, internal,} and \emph{leaf} vertices to each tree $\tau_j$.
\begin{itemize}
 \item A \emph{root} vertex labeled as $W_j$, $j=1,2,\cdots,k$, to represent $J^1_j(x_j;x'_j)$.
 \item An \emph{internal} vertex labeled by $v_\ell$, $\ell=1,2,\cdots,n$, corresponding to $B_{\sigma(k+2\ell-1);k+2\ell-1,k+2\ell}$ and attached to the time variable $t_\ell$.
 \item A \emph{leaf} vertex $u_i$, $i=1,2,\cdots,k+2n$, representing each factor $(\Ket{\phi}\Bra{\phi})(x_i;x'_i)$.
\end{itemize}

Next, we connect the vertices with \emph{edges}, as described below.
\begin{itemize}
 \item If $v_\ell$ is the smallest value of $\ell$ such that $\sigma(k+2\ell-1)=j$, then we connect $v_\ell$ to the root vertex $W_j$ and write $W_j\sim v_\ell$ (or equivalently $W_j\sim B_{\sigma(k+2\ell-1);k+2\ell-1,k+2\ell}$).  If there is no internal vertex connected to a root vertex $W_j$, then we connect $W_j$ to the leaf $u_j$, and write $W_j \sim u_j$.
 \item For any $1<\ell\leq n$, if $\exists \ell'>l$ such that $\sigma(k+2\ell-1)=\sigma(k+2\ell'-1)$ or $\sigma(k+2\ell'-1)=k+2\ell-1$, then we connect $v_{\ell}$ and $v_{\ell'}$ and write $v_{\ell} \sim v_{\ell'}$ (or equivalently $B_{\sigma(k+2\ell-1);k+2\ell-1,k+2\ell}\sim B_{\sigma(k+2\ell'-1);k+2\ell'-1,k+2\ell'}$). In this case, we call $v_\ell$ the \emph{parent vertex} of $v_{\ell'}$, and $v_{\ell'}$ the \emph{child vertex} of $v_\ell$. We denote the three child vertices of $v_\ell$ by $v_{k_{-}(\ell)}$, $v_{k(\ell)}$ and $v_{k_{+}(\ell)}$, with $k_{-}(\ell)<k(\ell)<k_{+}(\ell)$. 
 
 \item When there is no internal vertex with $\ell'>l$ and $k+2\ell-1=\sigma(k+2\ell'-1)$, we connect $v_{\ell}$ to the leaf vertices $u_{k+2\ell-1}$, $u_{k+2\ell}$ and write $v_{\ell} \sim (u_{k+2\ell-1}, u_{k+2\ell})$ (or equivalently $B_{\sigma(k+2\ell-1);k+2\ell-1,k+2\ell}\sim (u_{k+2\ell-1},u_{k+2\ell})$). \\
\end{itemize}

We remark that it follows from the construction above that each root vertex has only one child vertex, and each internal vertex has exactly three child vertices (which can be either internal and leaf). We call the tree $\tau_j$ \emph{distinguished}  if $v_n \in \tau_j$, and \emph{regular} if $v_n\notin \tau_j$. The three leaves connected to $v_n$ are called \emph{distinguished leaf vertices}, and all other leaves are called \emph{regular leaf vertices}.  Clearly, there are $k-1$ regular trees and one distinguished tree in each tree graph. \\

A sample tree graph is given in Figure \ref{Figure binary_tree}, for $J^k$ as in \eqref{Jk example}. Each tree $\tau_j$ has root vertex $W_j$, for $j=1,2$. The leaf vertices $u_1, u_3, u_4, u_7, u_8, u_9, u_{10}$ and the internal vertices $v_1, v_3, v_4$ (or $B_{1;3,4}, B_{4;7,8}, B_{4;9,10}$) are distinguished. $\tau_1$ is the distinguished tree, and is drawn with thick edges. Tree $\tau_2$ with vertices $W_2, v_2, u_2, u_5, u_6$ is the regular tree, and is drawn with thin edges. \\

\begin{figure}
\centering
\def\svgwidth{3.5in}
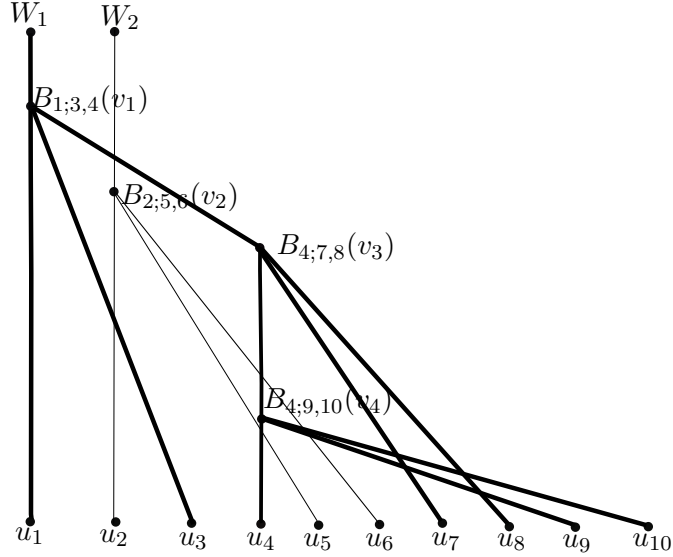
\caption{An example tree graph for $J^k$. It is a disjoint union of two trees $\tau_1$ and $\tau_2$ with root vertices $W_1$ and $W_2$, respectively.  Each tree corresponds to a one-particle kernel in the example in section \ref{sec: example factorization}, where $k=2$ and $n=4$.}\label{Figure binary_tree}
\end{figure}

\subsection{The distinguished one particle kernel $J^1_j$}
Let $\tau_j$ denote the distinguished tree graph.  It has $m_j$ internal vertices $(v_{\ell_j,\alpha})_{\alpha=1}^{m_j}$ and $2m_j+1$ leaf vertices $(u_{j,i})_{i=1}^{2m_j+1}$.  We enumerate the internal vertices with $\alpha\in\{1,\dots,m_j\}$ and the leaf vertices with $\alpha\in\{m_j+1,\dots,3m_j+1\}$.  To simplify notation, we refer to the vertex $v_{j,\alpha}$ by its label $\alpha$.  We observe that $J_j^1$ has the form
\begin{align}
&J_j^1(t,t_{\ell_{j,1}},\dots,t_{\ell_j,m_j};\sigma_j)\label{factorized J unsimplified}\\
&=U^{(1)}(t-t_1)\cdots U^{(1)}(t_{\ell_{j,1}-1}-t_{\ell_{j,1}})B_{\sigma_j(2);2,3}\cdots\nonumber\\
&\hspace{1cm}\cdots B_{\sigma_j(2\alpha-2);2\alpha-2,2\alpha-1}U^{(2\alpha-1)}(t_{\ell_{j,\alpha-1}}-t_{\ell_{j,\alpha-1}+1})\cdots U^{(2\alpha-1)}(t_{\ell_{j,\alpha}-1}-t_{\ell_{j,\alpha}})B_{\sigma_j(2\alpha);2\alpha,2\alpha+1}\cdots\nonumber\\
&\hspace{1cm}\cdots U^{(2m_j-1)}(t_{\ell_j,m_j-1}-t_{\ell_j,m_j})B_{\sigma_j(2m_j),2m_j,2m_j+1}(\pphi)^{\otimes(2m_j+1)}.\nonumber
\end{align}
By the semigroup property
\begin{align*}
U^{(\alpha)}(t)U^{(\alpha)}(s)=U^{(\alpha)}(t+s),
\end{align*}
and the fact that $\sigma_j(2)=1$, \eqref{factorized J unsimplified} reduces to
\begin{align}
&J_j^1(t,t_{\ell_{j,1}},\dots,t_{\ell_j,m_j};\sigma_j)\label{factorized J}\\
&=U^{(1)}(t-t_{\ell_{j,1}})B_{1;2,3}\cdots\nonumber\\
&\hspace{1cm}\cdots B_{\sigma_j(2\alpha-2);2\alpha-2,2\alpha-1}U^{(2\alpha-1)}(t_{\ell_{j,\alpha-1}}-t_{\ell_{j,\alpha}})B_{\sigma_j(2\alpha);2\alpha,2\alpha+1}\cdots\nonumber\\
&\hspace{1cm}\cdots U^{(2m_j-1)}(t_{\ell_j,m_j-1}-t_{\ell_j,m_j})B_{\sigma_j(2m_j);2m_j,2m_j+1}(\pphi)^{\otimes(2m_j+1)},\nonumber
\end{align}
where $\ell_{j,m_j}=n$. \\

\subsection{Definition of the kernels $\Theta_\alpha$ at the vertices of the distinguished tree graph}

In this section, we proceed as in \cite{CHPS}, and recursively assign a kernel $\Theta_\alpha$ to each vertex $\alpha$ of the distinguished tree graph.  The kernels at the vertices of the regular tree graph are defined similarly.  We begin by assigning the kernel
\begin{align*}
\Theta_\alpha(x;x'):=\phi(x)\overline{\phi}(x')
\end{align*}
to the leaf vertex with label $\alpha\in\{m_j+1,\dots,3m_j+1\}$. \\

Next, we determine $\Theta_{m_j}$ at the distinguished vertex $\alpha=m_j$ from the term on the last line of \eqref{factorized J}, given by
\begin{align*}
B_{\sigma_j(2m_j);2m_j,2m_j+1}(\pphi)^{\otimes (2m_j+1)}
&=(\pphi)^{\otimes(\sigma_j(2m_j)-1)}\otimes\Theta_{m_j} \otimes(\pphi)^{\otimes(2m_j+1-\sigma_j(2m_j)-2)}
\end{align*}
where
\begin{align}
\Theta_{m_j}(x;x'):=\tilde{\psi}(x)\overline{\phi}(x')-\phi(x)\overline{\tilde{\psi}}(x')\label{theta_mj}
\end{align}
with $\tilde{\psi}:=|\phi|^4\phi$.  It is obtained from contracting three copies of $\pphi$ at the three leaf vertices $\kappa_-(m_j),\kappa(m_j),\kappa_+(m_j)$ which have $m_j$ as their parent vertex.

Now we are ready to begin the induction.  Let $\alpha\in\{1,\dots,m_j-1\}$.  Suppose that the kernels $\Theta_{\alpha'}$ have been determined for all $\alpha'>\alpha$.  We let $\kappa_-(\alpha),\kappa(\alpha),\kappa_+(\alpha)$ label the three child vertices (of internal or leaf type) of $\alpha$.
Since $\Theta_{\kappa_-(\alpha)},\Theta_{\kappa(\alpha)},$ and $\Theta_{\kappa_+(\alpha)}$ have already been determined, we can now define
\begin{align*}
&\Theta_{\alpha}(x;x')\\
&=B_{1;2,3}((U^{(1)}(t_\alpha-t_{\kappa_-(\alpha)})\Theta_{\kappa_-(\alpha)})\otimes(U^{(1)}(t_\alpha-t_{\kappa(\alpha)})\Theta_{\kappa(\alpha)})\otimes(U^{(1)}(t_\alpha-t_{\kappa_+(\alpha)})\Theta_{\kappa_+(\alpha)}))(x;x').\\
\end{align*}

The induction ends when we obtain the kernel $\Theta_1$ at $\alpha=1$.

\subsection{Key properties of the kernels $\Theta_\alpha$}\label{key properties}  As in \cite{CHPS}, we observe that the kernels $\Theta_\alpha$ satisfy the following properties.
\begin{itemize}
\item $\Theta_\alpha$ can be written as a sum of differences of factorized kernels
\begin{align}
\Theta_\alpha(x;x')=\sum_{\beta_\alpha}c_{\beta_\alpha}^\alpha\chi_{\beta_\alpha}^\alpha(x)\overline{\psi_{\beta_\alpha}^\alpha}(x')\label{factorized}
\end{align}
with at most $2^{m_j-\alpha}$ nonzero coefficients $c_{\beta_\alpha}^\alpha\in\{1,-1\}$.

\item The product $\chi_{\beta_\alpha}^\alpha(x)\overline{\psi_{\beta_\alpha}^\alpha}(x')$ in \eqref{factorized} above is either of the form
\begin{align}
\chi_{\beta_\alpha}^\alpha(x)\overline{\psi_{\beta_\alpha}^\alpha}(x')
&=(U_{\alpha;\kappa_-(\alpha)}\chi_{\beta_{\kappa_-(\alpha)}}^{\kappa_-(\alpha)})(x)\overline{(U_{\alpha;\kappa_-(\alpha)}\psi_{\beta_{\kappa_-(\alpha)}}^{\kappa_-(\alpha)})}(x')\nonumber\\
&\hspace{1cm}A\bigg[V_\infty,(U_{\alpha;\kappa(\alpha)}\chi_{\beta_{\kappa(\alpha)}}^{\kappa(\alpha)})\overline{(U_{\alpha;\kappa(\alpha)}\psi_{\beta_{\kappa(\alpha)}}^{\kappa(\alpha)})},  \notag \\
&\hspace{1.5cm}(U_{\alpha;\kappa_+(\alpha)}\chi_{\beta_{\kappa_+(\alpha)}}^{\kappa_+(\alpha)})\overline{(U_{\alpha;\kappa_+(\alpha)}\psi_{\beta_{\kappa_+(\alpha)}}^{\kappa_+(\alpha)})}\bigg](x)\label{form1}
\end{align}
or
\begin{align}
\chi_{\beta_\alpha}^\alpha(x)\overline{\psi_{\beta_\alpha}^\alpha}(x')
&=(U_{\alpha;\kappa_-(\alpha)}\chi_{\beta_{\kappa_-(\alpha)}}^{\kappa_-(\alpha)})(x)\overline{(U_{\alpha;\kappa_-(\alpha)}\psi_{\beta_{\kappa_-(\alpha)}}^{\kappa_-(\alpha)})}(x')\nonumber\\
&\hspace{1cm}A\bigg[V_\infty,(U_{\alpha;\kappa(\alpha)}\chi_{\beta_{\kappa(\alpha)}}^{\kappa(\alpha)})\overline{(U_{\alpha;\kappa(\alpha)}\psi_{\beta_{\kappa(\alpha)}}^{\kappa(\alpha)})}, \notag \\
&\hspace{1.5cm}(U_{\alpha;\kappa_+(\alpha)}\chi_{\beta_{\kappa_+(\alpha)}}^{\kappa_+(\alpha)})\overline{(U_{\alpha;\kappa_+(\alpha)}\psi_{\beta_{\kappa_+(\alpha)}}^{\kappa_+(\alpha)})}\bigg](x')
\label{form2}
\end{align}
for some values of $\beta_{\kappa_-(\alpha)},\beta_{\kappa(\alpha)},\beta_{\kappa_+(\alpha)}$ that depend on $\beta_\alpha$.
The trilinear operator $A$ is defines as
\begin{equation}   \label{definition of A}
A[V_\infty,f,g](x):=\int\int V_\infty(x-y_1,x-y_2)f(y_1)g(y_2)\,dy_1\,dy_2.
\end{equation}

Observe that above, the function $\chi_{\beta_\alpha}^\alpha$ is either of the quintic form
\begin{align}
\chi_{\beta_\alpha}^\alpha(x)
&=(U_{\alpha;\kappa_-(\alpha)}\chi_{\beta_{\kappa_-(\alpha)}}^{\kappa_-(\alpha)})(x)\nonumber\\
&\hspace{1cm}A\bigg[V_\infty,(U_{\alpha;\kappa(\alpha)}\chi_{\beta_{\kappa(\alpha)}}^{\kappa(\alpha)})\overline{(U_{\alpha;\kappa(\alpha)}\psi_{\beta_{\kappa(\alpha)}}^{\kappa(\alpha)})},\\
&\hspace{1.5cm}(U_{\alpha;\kappa_+(\alpha)}\chi_{\beta_{\kappa_+(\alpha)}}^{\kappa_+(\alpha)})\overline{(U_{\alpha;\kappa_+(\alpha)}\psi_{\beta_{\kappa_+(\alpha)}}^{\kappa_+(\alpha)})}\bigg](x)\label{cubic form}
\end{align}
or the linear form
\begin{align}
\chi_{\beta_\alpha}^\alpha(x)
&=(U_{\alpha;\kappa_-(\alpha)}\chi_{\beta_{\kappa_-(\alpha)}}^{\kappa_-(\alpha)})(x).\label{linear form}
\end{align}
Accordingly, $\psi_{\beta_\alpha}^\alpha$ respectively is either of linear or quintic form, and the product $\chi_{\beta_\alpha}^\alpha(x)\overline{\psi_{\beta_\alpha}^\alpha}(x')$ always has sextic form \eqref{form1} or \eqref{form2}.
\item We call the functions $\chi_{\beta_\alpha}^\alpha,\psi_{\beta_\alpha}^\alpha$ in the sum \eqref{factorized} \emph{distinguished} if they are a function of $|\phi|^4\phi$. In the product on the right hand side of \eqref{form1}, respectively \eqref{form2}, at most one of the six factors is distinguished.  Indeed, this is true for all regular leaf vertices, and for the distinguished vertex \eqref{theta_mj}.  By induction along decreasing values of $\alpha$, it is also true for the internal vertices.
\end{itemize}

As in \cite{CHPS}, we make the following assumption, which simplifies the notation without loss of generality.

\begin{hyp}
We assume that only the functions $\psi_{\beta_1}^1$ and $(\psi_{\beta_{\kappa_+^q(1)}}^{\kappa_+^q(1)})$ are distinguished, where we define
\begin{align*}
\kappa_+^q(1):=\underbrace{\kappa_+(\kappa_+(\dots(\kappa_+}_{q\text{ times}}(1))\dots)).
\end{align*}
\end{hyp}

\section{Proof of Proposition \ref{prop: zero trace norm}}    \label{sec: recursive bounds}

In this section, we prove Proposition \ref{prop: zero trace norm}.  To simplify notation, we denote the time variable $t_{\ell_j,\alpha}$ by $t_\alpha$.  We denote the subtree of $\tau_j$ with root at the vertex $\alpha$ by $\tau_{j,\alpha}$, and let
\begin{align*}
\int \bigg[\prod_{\alpha'\in\tau_{j,\alpha}}dt_{\alpha'}\bigg]:=\int_{[0,T)^{d_\alpha}}\bigg[\prod_{\alpha'\in\tau_{j,\alpha}}dt_{\alpha'}\bigg]
\end{align*}
be integration with respect to all time variables attached to the internal and root vertices of the subtree $\tau_{j,\alpha}$.  Here, the total number of internal and root vertices of the tree $\tau_{j,\alpha}$ is denoted by $d_\alpha$. 
\begin{lemma}\label{induction}
For $V_\infty\in L^\frac{1}{1-\epsilon}(\mathbb{R}^6)$ with small $\epsilon\ge 0$ (or $V_\infty(y,z)=\lambda\delta_0(y)\delta_0(z)$ with $\epsilon=0$ and $\|V_\infty\|_{L^1}:=\lambda$), we have the $\dot H^{-1}$ bound
\begin{align}
&\int \bigg[\prod_{\alpha'\in\tau_{j,\alpha}}dt_{\alpha'}\bigg]\|\psi_{\beta_\alpha}^\alpha\|_{\dot H^{-1}}\|\chi_{\beta_\alpha}^\alpha\|_{\dot H^1}\nonumber\\
&\le CT^{3\epsilon}\|V_\infty\|_{L^\frac{1}{1-\epsilon}}\int \bigg[\prod_{\alpha'\in\tau_{j,\kappa_-(\alpha)}}dt_{\alpha'}\bigg]\|\psi_{\beta_{\kappa_-(\alpha)}}^{\kappa_-(\alpha)}\|_{\dot H^1}\|\chi_{\beta_{\kappa_-(\alpha)}}^{\kappa_-(\alpha)}\|_{\dot H^1}\nonumber\\
&\hspace{1cm}\cdot \int \bigg[\prod_{\alpha'\in\tau_{j,\kappa(\alpha)}}dt_{\alpha'}\bigg]\|\psi_{\beta_{\kappa(\alpha)}}^{\kappa(\alpha)}\|_{\dot H^1}\|\chi_{\beta_{\kappa(\alpha)}}^{\kappa(\alpha)}\|_{\dot H^1}\nonumber\\
&\hspace{1cm}\cdot \int \bigg[\prod_{\alpha'\in\tau_{j,\kappa_+(\alpha)}}dt_{\alpha'}\bigg]\|\psi_{\beta_{\kappa_+(\alpha)}}^{\kappa_+(\alpha)}\|_{\dot H^{-1}}\|\chi_{\beta_{\kappa_+(\alpha)}}^{\kappa_+(\alpha)}\|_{\dot H^1}\label{H^{-1} induction}
\end{align}
and the $\dot H^1$ bound
\begin{align}
&\int \bigg[\prod_{\alpha'\in\tau_{j,\alpha}}dt_{\alpha'}\bigg]\|\psi_{\beta_\alpha}^\alpha\|_{\dot H^1}\|\chi_{\beta_\alpha}^\alpha\|_{\dot H^1}\nonumber\\
&\le CT^{3\epsilon}\|V_\infty\|_{L^\frac{1}{1-\epsilon}}\int \bigg[\prod_{\alpha'\in\tau_{j,\kappa_-(\alpha)}}dt_{\alpha'}\bigg]\|\psi_{\beta_{\kappa_-(\alpha)}}^{\kappa_-(\alpha)}\|_{\dot H^1}\|\chi_{\beta_{\kappa_-(\alpha)}}^{\kappa_-(\alpha)}\|_{\dot H^1}\nonumber\\
&\hspace{1cm}\cdot\int \bigg[\prod_{\alpha'\in\tau_{j,\kappa(\alpha)}}dt_{\alpha'}\bigg]\|\psi_{\beta_{\kappa(\alpha)}}^{\kappa(\alpha)}\|_{\dot H^1}\|\chi_{\beta_{\kappa(\alpha)}}^{\kappa(\alpha)}\|_{\dot H^1}\nonumber\\
&\hspace{1cm}\cdot\int \bigg[\prod_{\alpha'\in\tau_{j,\kappa_+(\alpha)}}dt_{\alpha'}\bigg]\|\psi_{\beta_{\kappa_+(\alpha)}}^{\kappa_+(\alpha)}\|_{\dot H^1}\|\chi_{\beta_{\kappa_+(\alpha)}}^{\kappa_+(\alpha)}\|_{\dot H^1}.\label{H^s induction} 
\end{align}

\begin{proof}
To prove \eqref{H^{-1} induction}, we apply the bound \eqref{inequality:multilinear hartree H^{-1}} (or \eqref{inequality:multilinear gp H^{-1}}) to \eqref{form1} and \eqref{form2} and obtain
\begin{align*}
&\int \bigg[\prod_{\alpha'\in\tau_{j,\alpha}}dt_{\alpha'}\bigg]\|\psi_{\beta_\alpha}^\alpha\|_{\dot H^{-1}}\|\chi_{\beta_\alpha}^\alpha\|_{\dot H^1}\\
&\le CT^{3\epsilon}\|V_\infty\|_{L^\frac{1}{1-\epsilon}}\int_{[0,T)^{d_{\alpha}-1}}\bigg[\prod_{\alpha'\in\tau_{j,\kappa_-(\alpha)}\cup\tau_{j,\kappa(\alpha)}\cup\tau_{j;\kappa_+(\alpha)}}dt_{\alpha'}\bigg]\|\chi_{\beta_{\kappa_-(\alpha)}}^{\kappa_-(\alpha)}\|_{\dot H^1}\\
&\hspace{1cm}\cdot \|\psi_{\beta_{\kappa_-(\alpha)}}^{\kappa_-(\alpha)}\|_{\dot H^1}\|\chi_{\beta_{\kappa(\alpha)}}^{\kappa(\alpha)}\|_{\dot H^1}\|\psi_{\beta_{\kappa(\alpha)}}^{\kappa(\alpha)}\|_{\dot H^{1}}\|\chi_{\beta_{\kappa_+(\alpha)}}^{\kappa_+(\alpha)}\|_{\dot H^1}\|\psi_{\beta_{\kappa_+(\alpha)}}^{\kappa_+(\alpha)}\|_{\dot H^{-1}}\\
&= CT^{3\epsilon}\|V_\infty\|_{L^\frac{1}{1-\epsilon}}\int_{[0,T)^{d_{\kappa_-(\alpha)}}}\bigg[\prod_{\alpha'\in\tau_{j,\kappa_-(\alpha)}}dt_{\alpha'}\bigg]\|\chi_{\beta_{\kappa_-(\alpha)}}^{\kappa_-(\alpha)}\|_{\dot H^1}\|\psi_{\beta_{\kappa_-(\alpha)}}^{\kappa_-(\alpha)}\|_{\dot H^1}\\
&\hspace{1cm}\cdot\int_{[0,T)^{d_{\kappa_-(\alpha)}}}\bigg[\prod_{\alpha'\in\tau_{j,\kappa(\alpha)}}dt_{\alpha'}\bigg]\|\chi_{\beta_{\kappa(\alpha)}}^{\kappa(\alpha)}\|_{\dot H^1}\|\psi_{\beta_{\kappa(\alpha)}}^{\kappa(\alpha)}\|_{\dot H^1}\\
&\hspace{1cm}\cdot\int_{[0,T)^{d_{\kappa_+(\alpha)}}}\bigg[\prod_{\alpha'\in\tau_{j,\kappa_+(\alpha)}}dt_{\alpha'}\bigg]\|\chi_{\beta_{\kappa_+(\alpha)}}^{\kappa_+(\alpha)}\|_{\dot H^1}\|\psi_{\beta_{\kappa_+(\alpha)}}^{\kappa_+(\alpha)}\|_{\dot H^{-1}}.
\end{align*}
In the second step, we performed the $t_\alpha$ integral.  In the second step, we used the fact that the terms $\psi_{\beta_\alpha}^\alpha,\chi_{\beta_\alpha}^\alpha$ depend only on the time variables $t_{\alpha'}$ attached to the vertices of the subtree $\tau_{j,\alpha}$. \\

To prove \eqref{H^s induction}, we apply the bound \eqref{inequality:multilinear hartree H^1} (or \eqref{inequality:multilinear gp H^1})  to \eqref{form1} and \eqref{form2} and obtain
\begin{align*}
&\int \bigg[\prod_{\alpha'\in\tau_{j,\alpha}}dt_{\alpha'}\bigg]\|\psi_{\beta_\alpha}^\alpha\|_{\dot H^{1}}\|\chi_{\beta_\alpha}^\alpha\|_{\dot H^1}\\
&\le CT^{3\epsilon}\|V_\infty\|_{L^\frac{1}{1-\epsilon}}\int_{[0,T)^{d_{\alpha}-1}}\bigg[\prod_{\alpha'\in\tau_{j,\kappa_-(\alpha)}\cup\tau_{j,\kappa(\alpha)}\cup\tau_{j;\kappa_+(\alpha)}}dt_{\alpha'}\bigg]\|\chi_{\beta_{\kappa_-(\alpha)}}^{\kappa_-(\alpha)}\|_{\dot H^1}\\
&\hspace{1cm} \cdot \|\psi_{\beta_{\kappa_-(\alpha)}}^{\kappa_-(\alpha)}\|_{\dot H^1}\|\chi_{\beta_{\kappa(\alpha)}}^{\kappa(\alpha)}\|_{\dot H^1}\|\psi_{\beta_{\kappa(\alpha)}}^{\kappa(\alpha)}\|_{\dot H^{1}}\|\chi_{\beta_{\kappa_+(\alpha)}}^{\kappa_+(\alpha)}\|_{\dot H^1}\|\psi_{\beta_{\kappa_+(\alpha)}}^{\kappa_+(\alpha)}\|_{\dot H^{1}}\\
&= CT^{3\epsilon}\|V_\infty\|_{L^\frac{1}{1-\epsilon}}\int_{[0,T)^{d_{\kappa_-(\alpha)}}}\bigg[\prod_{\alpha'\in\tau_{j,\kappa_-(\alpha)}}dt_{\alpha'}\bigg]\|\chi_{\beta_{\kappa_-(\alpha)}}^{\kappa_-(\alpha)}\|_{\dot H^1}\|\psi_{\beta_{\kappa_-(\alpha)}}^{\kappa_-(\alpha)}\|_{\dot H^1}\\
&\hspace{1cm} \cdot \int_{[0,T)^{d_{\kappa_-(\alpha)}}}\bigg[\prod_{\alpha'\in\tau_{j,\kappa(\alpha)}}dt_{\alpha'}\bigg]\|\chi_{\beta_{\kappa(\alpha)}}^{\kappa(\alpha)}\|_{\dot H^1}\|\psi_{\beta_{\kappa(\alpha)}}^{\kappa(\alpha)}\|_{\dot H^1}\\
&\hspace{1cm} \cdot \int_{[0,T)^{d_{\kappa_+(\alpha)}}}\bigg[\prod_{\alpha'\in\tau_{j,\kappa_+(\alpha)}}dt_{\alpha'}\bigg]\|\chi_{\beta_{\kappa_+(\alpha)}}^{\kappa_+(\alpha)}\|_{\dot H^1}\|\psi_{\beta_{\kappa_+(\alpha)}}^{\kappa_+(\alpha)}\|_{\dot H^{1}}.
\end{align*}
\end{proof}
\end{lemma}

We now recursively apply the bounds in the statement Lemma \ref{induction} to conclude the proof of uniqueness of solutions to the quintic GP and Hartree hierarchy.

\begin{prop}\label{H^{-1} integral}
For the distinguished tree $\tau_j$, we have the bound
\begin{align}
&\int_{[0,T)^{m_j-1}}dt_1\dots dt_{m_j-1}{\rm Tr}\bigg(\,\bigg|R^{(1,-1)}J^1_j(t,t_1,\cdots,t_{m_j};\sigma_j)\bigg|\,\bigg)\nonumber\\
&\hspace{3cm}\le 2^{m_j}C^{m_j-1}T^{3\epsilon (m_j-1)}\|V_\infty\|_{L^\frac{1}{1-\epsilon}}^{m_j-1}\|\phi\|_{\dot H^{1}}^{4m_j-3}\|A[V_\infty,|\phi|^2,|\phi|^2]\phi\|_{\dot H^{-1}}.\label{distinguished bound}
\end{align}
\begin{proof}
\begin{align}
&\int_{[0,T)^{m_j-1}}dt_1\dots dt_{m_j-1}{\rm Tr}\bigg(\,\bigg|R^{(1,-1)}J^1_j(t,t_1,\cdots,t_{m_j};\sigma_j)\bigg|\,\bigg)\nonumber\\
&=\int_{[0,T)^{m_j-1}}dt_1\dots dt_{m_j-1}{\rm Tr}\bigg(\,\bigg|R^{(1,-1)}U^{(1)}(t-t_1)\Theta_1\bigg|\,\bigg)\nonumber\\
&\le\sum_{\beta_1}\int_{[0,T)^{m_j-1}}dt_1\cdots dt_{m_j-1}\|\psi_{\beta_1}^1\|_{\dot H^{-1}}\|\chi_{\beta_1}^1\|_{\dot H^{-1}}\nonumber\\
&\le\sum_{\beta_1}\int_{[0,T)^{m_j-1}}dt_1\cdots dt_{m_j-1}\|\psi_{\beta_1}^1\|_{\dot H^{-1}}\|\chi_{\beta_1}^1\|_{\dot H^{1}}\nonumber\\
&\le\sum_{\beta_{\kappa_-(1)},\beta_{\kappa(1)},\beta_{\kappa_+(1)}}CT^{3\epsilon}\|V_\infty\|_{L^\frac{1}{1-\epsilon}}\int_{[0,T)^{d_{\kappa_-(\alpha)}}}\bigg[\prod_{\alpha'\in\tau_{j,\kappa_-(\alpha)}}dt_{\alpha'}\bigg]\|\chi_{\beta_{\kappa_-(\alpha)}}^{\kappa_-(\alpha)}\|_{\dot H^{1}}\|\psi_{\beta_{\kappa_-(\alpha)}}^{\kappa_-(\alpha)}\|_{\dot H^{1}}\label{H^1 step 1}\\
&\hspace{1cm} \cdot \int_{[0,T)^{d_{\kappa(\alpha)}}}\bigg[\prod_{\alpha'\in\tau_{j,\kappa(\alpha)}}dt_{\alpha'}\bigg]\|\chi_{\beta_{\kappa(\alpha)}}^{\kappa(\alpha)}\|_{\dot H^1}\|\psi_{\beta_{\kappa(\alpha)}}^{\kappa(\alpha)}\|_{\dot H^{1}}\label{H^1 step 2}\\
&\hspace{1cm}  \cdot \int_{[0,T)^{d_{\kappa_+(\alpha)}}}\bigg[\prod_{\alpha'\in\tau_{j,\kappa_+(\alpha)}}dt_{\alpha'}\bigg]\|\chi_{\beta_{\kappa_+(\alpha)}}^{\kappa_+(\alpha)}\|_{\dot H^1}\|\psi_{\beta_{\kappa_+(\alpha)}}^{\kappa_+(\alpha)}\|_{\dot H^{-1}}\label{H^-1 step}
\end{align}
In the last step, we performed the $t_1$ integral using \eqref{H^{-1} induction}.  Now, to bound \eqref{H^1 step 1} and \eqref{H^1 step 2}, we iterate the $H^{1}$ bound \eqref{H^s induction}.  To bound \eqref{H^-1 step}, we iterate both \eqref{H^{-1} induction} and \eqref{H^s induction}.  This establishes \eqref{distinguished bound}.
\end{proof}
\end{prop}
\begin{prop}\label{H^1 integral}
For the regular tree $\tau_j$, we have the bound
\begin{align}
&\int_{[0,T)^{m_j}}dt_1\dots dt_{m_j}{\rm Tr}\bigg(\,\bigg|R^{(1,-1)}J^1_j(t,t_1,\cdots,t_{m_j};\sigma_j)\bigg|\,\bigg)\nonumber\\
&\hspace{3cm}\le 2^{m_j}C^{m_j}T^{3\epsilon m_j}\|\phi\|_{\dot H^{1}}^{4m_j+2}.\label{regular bound d ge 2}
\end{align}
\begin{proof}
\begin{align*}
&\int_{[0,T)^{m_j}}dt_1\dots dt_{m_j}{\rm Tr}\bigg(\,\bigg|R^{(1,-1)}J^1_j(t,t_1,\cdots,t_{m_j};\sigma_j)\bigg|\,\bigg)\\
&=\int_{[0,T)^{m_j}}dt_1\dots dt_{m_j}{\rm Tr}\bigg(\,\bigg|R^{(1,-1)}U^{(1)}(t-t_1)\Theta_1\bigg|\,\bigg)\\
&\le\sum_{\beta_1}\int_{[0,T)^{m_j}}dt_1\cdots dt_{m_j}\|\psi_{\beta_1}^1\|_{\dot H^{-1}}\|\chi_{\beta_1}^1\|_{\dot H^{-1}}\\
&\le\sum_{\beta_1}\int_{[0,T)^{m_j}}dt_1\cdots dt_{m_j}\|\psi_{\beta_1}^1\|_{\dot H^{1}}\|\chi_{\beta_1}^1\|_{\dot H^{1}}
\end{align*}
From here, we iterate the $\dot H^{1}$ bound \eqref{H^s induction} to obtain \eqref{regular bound d ge 2}.
\end{proof}
\end{prop}

\begin{lemma}\label{final bound}
Suppose that $V_\infty\in L^{\frac{1}{1-\epsilon}}$.  Then
\begin{align*}
\|A[V_\infty,|\phi|^2,|\phi|^2]\phi\|_{\dot H^{-1}}\lesssim
\begin{cases}
\|V_\infty\|_{L^1}\|\phi\|_{\dot H^1}^5, &\text{ if }\epsilon=0\\
\|V_\infty\|_{L^\frac{1}{1-\epsilon}}\|\phi\|_{H^1}^5, &\text{ if }\epsilon>0.
\end{cases}
\end{align*}
Notice that when $\epsilon>0$, we measure the norm of $\phi$ in the non-homogeneous Sobolev space $H^1$.
\begin{proof}
By Strichartz estimates, Sobolev embedding, and Theorem \ref{beckner theorem}, we have
\begin{align*}
&\|A[V_\infty,|\phi|^2,|\phi|^2]\phi\|_{\dot H^{-1}}\\
&\lesssim\|A[V_\infty,|\phi|^2,|\phi|^2]\phi\|_{L^{\frac{6}{5}}}\\
&\le\|A[V_\infty,|\phi|^2,|\phi|^2]\phi\|_{L^{\frac{3}{2}}}\|\phi\|_{L^6}\\
&\le\|V_\infty\|_{L^\frac{1}{1-\epsilon}}\||\phi|^2\|_{L^{\frac{3}{1+3\epsilon}}}^2\|\phi\|_{L^6}\\
&=\|V_\infty\|_{L^\frac{1}{1-\epsilon}}\|\phi\|_{L^{\frac{6}{1+3\epsilon}}}^4\|\phi\|_{L^6}\\
&\lesssim
\begin{cases}
\|V_\infty\|_{L^1}\|\phi\|_{\dot H^1}^5, &\text{ if }\epsilon=0\\
\|V_\infty\|_{L^\frac{1}{1-\epsilon}}\|\phi\|_{H^1}^5, &\text{ if }\epsilon>0.
\end{cases}\qedhere
\end{align*}
\end{proof}
\end{lemma}

We are now ready to conclude the proof of Proposition \ref{prop: zero trace norm}.

\begin{proof}[Proof of Proposition \ref{prop: zero trace norm}.]
Recall from \eqref{J decomposition} that $J^k$ can be decomposed into a product of $k$ one-particle kernels
\begin{align*}
J^k(t,t_1,\dots,t_n;\sigma)=\prod_{j=1}^k J^1_j(t,t_{\ell_{j,1}},\dots,t_{\ell_{j,m_j}};\sigma_j),
\end{align*}
where only one of the factors $J_j^1$ distinguished.  It now follows from Propositions \ref{H^{-1} integral} and \ref{H^1 integral} that
\begin{align*}
&\int_{[0,T)^{n-1}}dt_1\cdots dt_{n-1}{\rm Tr}\bigg(\bigg|R^{(k,-1)}J^k(t,t_1,\dots,t_n;\sigma)\bigg|\bigg)\\
&=\int_{[0,T)^{n-1}}dt_1\cdots dt_{n-1}\prod_{j=1}^k{\rm Tr}\bigg(\bigg|R^{(1,-1)}J^1_j(t,t_{\ell_{j,1}},\dots,t_{\ell_{j,m_j}};\sigma_j)\bigg|\bigg)\\
&\le
2^nC^{n-1}T^{3\epsilon(n-1)}\|V_\infty\|_{L^\frac{1}{1-\epsilon}}^{n-1}\|\phi\|_{\dot H^{1}}^{4(k+n)-5}\|A[V_\infty,|\phi|^2,|\phi|^2]\phi\|_{\dot H^{-1}}.
\end{align*}
Thus, by Lemma \ref{final bound}, the difference between two solutions $\gamma:=\gamma_1-\gamma_2$ satisfies
\begin{align*}
&{\rm Tr}|R^{(k,-1)}\gamma^{(k)}|\\
&\le (\#\mathcal{M}_{k,n})\sup_{\sigma\in\mathcal{M}_{k,n}}\sup_{i=1,2}\int_{[0,T)^n} d\underline{t}_n \int d\mu^{(i)}_{t_n}(\phi) \textup{Tr}(|R^{(k,-1)}J^k(\underline{t}_n;\sigma)|)\\
&\le
\bigg(CT^{3\epsilon}\|V_\infty\|_{L^\frac{1}{1-\epsilon}}\bigg)^{n-1}\int_0^T dt_n \int d\mu_{t_n}^{(i)}(\phi) \|\phi\|_{\dot H^{1}}^{4(k+n)-5}\|A[V_\infty,|\phi|^2,|\phi|^2]\phi\|_{\dot H^{-1}}\\
&\le
\begin{cases}
\bigg(C\|V_\infty\|_{L^1}\bigg)^n \int_0^T dt_n \int d\mu_{t_n}^{(i)}(\phi) \|\phi\|_{\dot H^{1}}^{4(k+n)},&\text{ if }\epsilon=0\\
\bigg(CT^{3\epsilon}\|V_\infty\|_{L^\frac{1}{1-\epsilon}}\bigg)^{n-1}\|V_\infty\|_{L^\frac{1}{1-\epsilon}}\int_0^T dt_n \int d\mu_{t_n}^{(i)}(\phi) \|\phi\|_{H^{1}}^{4(k+n)},&\text{ if }\epsilon>0
\end{cases}\\
&\le
\begin{cases}
\bigg(C\|V_\infty\|_{L^1}\bigg)^n TM^{4(k+n)},&\text{ if }\epsilon=0\\
\bigg(CT^{3\epsilon}\|V_\infty\|_{L^\frac{1}{1-\epsilon}}\bigg)^{n-1}\|V_\infty\|_{L^\frac{1}{1-\epsilon}}TM^{4(k+n)},&\text{ if }\epsilon>0
\end{cases}\\
&\rightarrow 0\text{ as }n\rightarrow\infty
\end{align*}
for $T$ sufficiently small if $\epsilon>0$, and for $M$ sufficiently small if $\epsilon=0$.  Thus ${\rm Tr}|R^{(k,-1)}\gamma^{(k)}|=0$.  Combining this with the a-priori bound
\begin{align*}
\begin{cases}
{\rm Tr}|R^{(k,1)}\gamma^{(k)}|<M^{2k},&\text{ if }\epsilon=0\\
{\rm Tr}|S^{(k,1)}\gamma^{(k)}|<M^{2k},&\text{ if }\epsilon>0
\end{cases}
\end{align*}
yields the desired result.  Namely,
\begin{align*}
\begin{cases}
{\rm Tr}|R^{(k,1)}\gamma^{(k)}|=0,&\text{ if }\epsilon=0\\
{\rm Tr}|S^{(k,1)}\gamma^{(k)}|=0,&\text{ if }\epsilon>0.
\end{cases}&\qedhere
\end{align*}
\end{proof}

\appendix

\section{Multilinear Estimates}   \label{sec: multilin estimates}

In this section, we present the key multilinear estimates that we will use to prove our main theorems. For the GP hierarchy, our key estimates are in Proposition \ref{prop:multilinear gp}.  The key estimates for the Hartree hierarchy are in Propositions \ref{prop:multilinear hartree}.

\begin{prop}[Multilinear estimates for GP]\label{prop:multilinear gp}
\begin{align}
\|(e^{it\Delta}f_1)(e^{it\Delta}f_2)(e^{it\Delta}f_3)(e^{it\Delta}f_4)(e^{it\Delta}f_5)\|_{L^1_t\dot H^{-1}_x}&\lesssim\|f_1\|_{\dot H^{-1}}\prod_{j=2}^5\|f_j\|_{\dot H^1}\label{inequality:multilinear gp H^{-1}},\\
\|(e^{it\Delta}f_1)(e^{it\Delta}f_2)(e^{it\Delta}f_3)(e^{it\Delta}f_4)(e^{it\Delta}f_5)\|_{L^1_t\dot H^{1}_x}&\lesssim\prod_{j=1}^5\|f_j\|_{\dot H^1}.\label{inequality:multilinear gp H^1}
\end{align}

For the proof, we need
\begin{lemma}[Negative Sobolev norm estimate]\label{split}
$$\|fg\|_{\dot H^{-1}}\lesssim\|f\|_{\dot W^{-1,6}}\|g\|_{\dot W^{1,\frac{3}{2}}}.$$
\end{lemma}

\begin{proof}
We prove the lemma by the standard duality argument, the product rule and the Sobolev inequality.
\begin{align*}
\int fg\overline{h}\,dx
&\le\|f\|_{\dot W^{-1,6}}\|gh\|_{\dot W^{1,\frac{6}{5}}}\\
&\lesssim\|f\|_{\dot W^{-1,6}}\bigg(\|g\|_{L^3}\|h\|_{\dot H^1}+\|g\|_{\dot W^{1,\frac{3}{2}}}\|h\|_{L^6}\bigg)\\
&\lesssim\|f\|_{\dot W^{-1,6}}\|g\|_{\dot W^{1,\frac{3}{2}}}\|h\|_{H^1}.
\end{align*}
\end{proof}

\begin{proof}
By Lemma \ref{split}, Sobolev embedding and Strichartz estimates, we prove that
\begin{align*}
&\|(e^{it\Delta}f_1)(e^{it\Delta}f_2)(e^{it\Delta}f_3)(e^{it\Delta}f_4)(e^{it\Delta}f_5)\|_{L^1_t\dot H^{-1}_x}\\
&\lesssim\|e^{it\Delta}f_1\|_{L^2_tW^{-1,6}_x}\bigg\|\prod_{j=2}^5e^{it\Delta}f_j\bigg\|_{L^2_t\dot W^{1,\frac{3}{2}}_x}\\
&\lesssim\|f_1\|_{\dot H^{-1}}\bigg(\|e^{it\Delta}f_2\|_{L^2_t\dot W^{1,6}_x}\prod_{j=3}^5\|e^{it\Delta}f_j\|_{L^\infty_tL^6_x}+\text{three similar terms (by the product rule)}\bigg)\\
&\lesssim\|f_1\|_{\dot H^{-1}}\prod_{j=2}^5\|f_j\|_{\dot H^1}
\end{align*}
and
\begin{align*}
&\|(e^{it\Delta}f_1)(e^{it\Delta}f_2)(e^{it\Delta}f_3)(e^{it\Delta}f_4)(e^{it\Delta}f_5)\|_{L^1_t\dot H^1_x}\\
&\lesssim\|e^{it\Delta}f_1\|_{L^2_t\dot W^{1,6}_x}\prod_{j=2}^5\|e^{it\Delta}f_j\|_{L^8_tL^{12}_x}+\text{ four similar terms (by the product rule)}\\
&\lesssim\|e^{it\Delta}f_1\|_{L^2_t\dot W^{1,6}_x}\prod_{j=2}^5\|e^{it\Delta}f_j\|_{L^8_t\dot W^{1,\frac{12}{5}}_x}+\text{ four similar terms (by the product rule)}\\
&\lesssim\prod_{j=1}^5\|f_j\|_{\dot H^1}.
\end{align*}
\end{proof}
\end{prop}

Recall the definition of the the trilinear operator $A$ in \eqref{definition of A}
\begin{equation*}  
A[V_\infty,f,g](x):=\int\int V_\infty(x-y_1,x-y_2)f(y_1)g(y_2)\,dy_1\,dy_2.
\end{equation*}

As an analogue of Proposition \ref{prop:multilinear gp}, we prove:
\begin{prop}[Multilinear estimates for Hartree]\label{prop:multilinear hartree}
Let $\epsilon\geq0$. Then, we have
\begin{equation}
\begin{aligned}
&\|A[V_\infty,(e^{it\Delta}f_1e^{it\Delta}f_2),(e^{it\Delta}f_3e^{it\Delta}f_4)]\cdot(e^{it\Delta}f_5)\|_{L^1_t\dot H^{-1}_x}\\
&\lesssim T^{3\epsilon}\|V_\infty\|_{L^{\frac{1}{1-\epsilon}}}\|f_m\|_{\dot H^{-1}}\prod_{\substack{\ell=1\\\ell\neq m}}^5\|f_\ell\|_{\dot H^1}\label{inequality:multilinear hartree H^{-1}},\quad\forall m=1,\cdots,5,
\end{aligned}
\end{equation}
and
\begin{equation}
\|A[V_\infty,(e^{it\Delta}f_1e^{it\Delta}f_2),(e^{it\Delta}f_3e^{it\Delta}f_4)]\cdot(e^{it\Delta}f_5)\|_{L^1_t\dot H^{1}_x}\lesssim T^{3\epsilon}\|V_\infty\|_{L^{\frac{1}{1-\epsilon}}}\prod_{\ell=1}^5\|f_\ell\|_{\dot H^1}. \label{inequality:multilinear hartree H^1}
\end{equation}
\end{prop}

We recall the convolution estimates in Beckner \cite{beckner}.
\begin{thm}\label{beckner theorem}
For $1<p<q<\infty, 1<s_k<p'/q', k=1,2$ and $1/q+2/p'=\sum 1/s_k, 2<p'/q'$,
\begin{align}
\|A[V_\infty,f,g]\|_{L^q(\mathbb{R}^d)}\le \|V_\infty\|_{L^p(\mathbb{R}^{2d})}\|f\|_{L^{s_1}(\mathbb{R}^d)}\|g\|_{L^{s_2}(\mathbb{R}^d)}.\label{beckner theorem equation}
\end{align}
\end{thm}

We note that Theorem \ref{beckner theorem} also holds for $p=1$.  Indeed, by the change of variables $(x-y,x-z)\rightarrow(y,z)$, Minkowski's inequality, and H\"older's inequality, we have
\begin{align*}
\|A[V_\infty,f,g]\|_{L^{q}}
&=\bigg\|\int\int V_\infty(y,z)f(x-y)g(x-z)\,dy\,dz\bigg\|_{L^{q}_x}\\
&\le\int\int |V_\infty(y,z)|\,\|f(x-y)g(x-z)\|_{L^{q}_x}\,dy\,dz\\
&\le\int\int |V_\infty(y,z)|\,\|f(x-y)\|_{L^{s_1}_x}\|g(x-z)\|_{L^{s_2}_x}\,dy\,dz\\
&=\|V_\infty\|_{L^1}\|f\|_{L^{s_1}}\|g\|_{L^{s_2}}.
\end{align*}


\begin{proof}[Proof of \eqref{inequality:multilinear hartree H^1}]
For $j\in\{1,2,3\}$, we have
\begin{align*}
&\bigg\|\partial_j\bigg[A[V_\infty,(e^{it\Delta}f_1e^{it\Delta}f_2),(e^{it\Delta}f_3e^{it\Delta}f_4)]\cdot(e^{it\Delta}f_5)\bigg]\bigg\|_{L^1_tL^2_x}\\
&\le\|A[V_\infty,(\partial_j e^{it\Delta}f_1e^{it\Delta}f_2),(e^{it\Delta}f_3e^{it\Delta}f_4)]\cdot(e^{it\Delta}f_5)\|_{L^1_tL^2_x}\\
&\hspace{1cm}+\text{ four similar terms (by the product rule)}\\
&=:I_1+I_2+I_3+I_4+I_5.
\end{align*}
By Theorem \ref{beckner theorem}, Strichartz estimates, and Sobolev embedding,
\begin{align*}
I_1&\le\bigg\|\|A[V_\infty,(\partial_j e^{it\Delta}f_1e^{it\Delta}f_2),(e^{it\Delta}f_3e^{it\Delta}f_4)]\|_{L^\frac{12}{5}_x}\|(e^{it\Delta}f_5)\|_{L^{12}_x}\bigg\|_{L^1_t}\\
&\lesssim\|V_\infty\|_{L^\frac{1}{1-\epsilon}}\bigg\|\|\partial_j e^{it\Delta}f_1e^{it\Delta}f_2\|_{L^\frac{4}{1+8\epsilon}_x}\|e^{it\Delta}f_3e^{it\Delta}f_4\|_{L^6_x}\|(e^{it\Delta}f_5)\|_{L^{12}_x}\bigg\|_{L^1_t}\\
&\le T^{3\epsilon}\|V_\infty\|_{L^\frac{1}{1-\epsilon}}\|\partial_j e^{it\Delta}f_1\|_{L^\frac{2}{1-6\epsilon}_tL^{\frac{6}{1+12\epsilon}}_x}\prod_{\ell=2}^5\|e^{it\Delta}f_\ell\|_{L^8_tL^{12}_x}\\
&\lesssim T^{3\epsilon}\|V_\infty\|_{L^\frac{1}{1-\epsilon}}\|\partial_j e^{it\Delta}f_1\|_{L^\frac{2}{1-6\epsilon}_tL^{\frac{6}{1+12\epsilon}}_x}\prod_{\ell=2}^5\|e^{it\Delta}f_\ell\|_{L^8_t\dot W^{1,\frac{12}{5}}_x}\\
&\lesssim T^{3\epsilon}\|V_\infty\|_{L^{\frac{1}{1-\epsilon}}}\prod_{\ell=1}^5\|f_\ell\|_{\dot H^1}.
\end{align*}
and similarly for $k\in\{2,3,4\}$.  For $k=5$, we have
\begin{align*}
I_5&\le\bigg\|\|A[V_\infty,(e^{it\Delta}f_1e^{it\Delta}f_2),(e^{it\Delta}f_3e^{it\Delta}f_4)]\|_{L^3_x}\|\partial_je^{it\Delta}f_5\|_{L^{6}_x}\bigg\|_{L^1_t}\\
&\lesssim\|V_\infty\|_{L^\frac{1}{1-\epsilon}}\bigg\|\|e^{it\Delta}f_1e^{it\Delta}f_2\|_{L^\frac{6}{1+12\epsilon}_x}\|e^{it\Delta}f_3e^{it\Delta}f_4\|_{L^6_x}\|\partial_je^{it\Delta}f_5\|_{L^{6}_x}\bigg\|_{L^1_t}\\
&\le T^{3\epsilon}\|V_\infty\|_{L^\frac{1}{1-\epsilon}}\|e^{it\Delta}f_1\|_{L^\frac{8}{1-24\epsilon}_tL^{\frac{12}{1+24\epsilon}}_x}\prod_{\ell=2}^4\|e^{it\Delta}f_\ell\|_{L^8_tL^{12}_x}\|\partial_j e^{it\Delta}f_5\|_{L^2_tL^{6}_x}\\
&\lesssim T^{3\epsilon}\|V_\infty\|_{L^\frac{1}{1-\epsilon}}\|e^{it\Delta}f_1\|_{L^\frac{8}{1-24\epsilon}_t\dot W^{\frac{12}{5+24\epsilon}}_x}\prod_{\ell=2}^4\|e^{it\Delta}f_\ell\|_{L^8_t\dot W^{1,\frac{12}{5}}_x}\|\partial_j e^{it\Delta}f_5\|_{L^2_tL^6_x}\\
&\lesssim T^{3\epsilon}\|V_\infty\|_{L^{\frac{1}{1-\epsilon}}}\prod_{\ell=1}^5\|f_\ell\|_{\dot H^1}.\qedhere
\end{align*}
\end{proof}

Before we proceeds to the proof of \eqref{inequality:multilinear hartree H^{-1}}, we define $\{P_1,P_2,P_3\}$ to be a conic decomposition of $\mathbb{R}^3$.  That is, $P_j$ is a Fourier multiplier with symbol $p_j:\mathbb{R}^3\rightarrow [0,1]$ such that for $\xi=(\xi_1,\xi_2,\xi_3)\in\mathbb{R}^3$,
\begin{align*}
&p_j(\xi)=1\text{ for }\xi_j^2\ge 2\sum_{j'\neq j}\xi_{j'}^2,\\
&p_j(\xi)=0\text{ for }\xi_j^2\le\frac{1}{2}\sum_{j'\neq j}\xi_{j'}^2,\text{ and}\\
&\sum_j p_j(\xi)=1\text{ for all }\xi\in\mathbb{R}^3.
\end{align*}
Observe that $|\xi_j|\sim |\xi|$ on the support of $p_j$.\\

\begin{proof} [Proof of \eqref{inequality:multilinear hartree H^{-1}} when $m=5$]
For $h\in \dot{H}^1(\mathbb{R}^3)$, we have
\begin{align*}
&\int A[V_\infty,(e^{it\Delta}f_1e^{it\Delta}f_2),(e^{it\Delta}f_3e^{it\Delta}f_4)](x)(e^{it\Delta}f_5)(x)\overline{h}(x)\,dx\\
&=\sum_{j=1}^3\int\int\int V_\infty(y,z)\partial_j\bigg[(e^{it\Delta}f_1e^{it\Delta}f_2)(x-y)(e^{it\Delta}f_3e^{it\Delta}f_4)(x-z)\overline{h}(x)\bigg]\\
&\hspace{6.5cm}\times(\partial_j^{-1}P_je^{it\Delta}f_5)(x)\,dy\,dz\,dx\\
&=\sum_{j=1}^3\int\int\int V_\infty(y,z)(\partial_je^{it\Delta}f_1e^{it\Delta}f_2)(x-y)(e^{it\Delta}f_3e^{it\Delta}f_4)(x-z)\\
&\hspace{6.5cm}\times\overline{h}(x)(\partial_j^{-1}P_je^{it\Delta}f_5)(x)\,dy\,dz\,dx\\
&\hspace{1cm}+\text{four similar terms (by the product rule)}\\
&=:I_1+I_2+I_3+I_4+I_5.
\end{align*}
By duality, it now suffices to show that
\begin{align}
\|I_k\|_{L^1_t}\lesssim T^{3\epsilon}\|V_\infty\|_{L^{\frac{1}{1-\epsilon}}}\|f_5\|_{\dot H^{-1}}\bigg(\prod_{\ell=1}^4\|f_\ell\|_{\dot H^1}\bigg)\|h\|_{\dot H^1}\label{I_j 1}
\end{align}
holds for $k\in\{1,2,3,4,5\}$.  By Theorem \ref{beckner theorem}, Strichartz estimates, and Sobolev embedding, we have
\begin{align*}
\|I_1\|_{L^1_t}&\le\sum_{j=1}^3\bigg\|\|\int\int V_\infty(y,z)(\partial_je^{it\Delta}f_1e^{it\Delta}f_2)(x-y)(e^{it\Delta}f_3e^{it\Delta}f_4)(x-z)\,dy\,dz\|_{L^{\frac{3}{2}}_x}\\
&\hspace{2cm}\times\|\partial_j^{-1}P_je^{it\Delta}f_5\|_{L^6_x}\|h\|_{L^6_x}\bigg\|_{L^1_t}\\
&\lesssim\sum_{j=1}^3\bigg\|\|V_\infty\|_{L^{\frac{1}{1-\epsilon}}}\|\partial_je^{it\Delta}f_1e^{it\Delta}f_2\|_{L^\frac{3}{1+6\epsilon}}\|e^{it\Delta}f_3e^{it\Delta}f_4\|_{L^3}\|\partial_j^{-1}P_je^{it\Delta}f_5\|_{L^6_x}\|h\|_{L^6_x}\bigg\|_{L^1_t}\\
&\lesssim\sum_{j=1}^3T^{3\epsilon}\|V_\infty\|_{L^{\frac{1}{1-\epsilon}}}\|\partial_je^{it\Delta}f_1\|_{L^\frac{2}{1-6\epsilon}_tL^\frac{6}{1+12\epsilon}_x}\prod_{\ell=2}^4\|e^{it\Delta}f_\ell\|_{L^\infty_tL^6_x}\|\partial_j^{-1}P_je^{it\Delta}f_5\|_{L^2_tL^6_x}\|h\|_{L^6_x}\\
&\lesssim T^{3\epsilon}\|V_\infty\|_{L^{\frac{1}{1-\epsilon}}}\|f_5\|_{\dot H^{-1}}\bigg(\prod_{\ell=1}^4\|f_\ell\|_{\dot H^1}\bigg)\|h\|_{\dot H^1},
\end{align*}
and similarly \eqref{I_j 1} holds for $k\in\{2,3,4\}$.  For $k=5$, we bound $\|I_5\|_{L^1_t}$ by
\begin{align*}
&\sum_{j=1}^3\bigg\|\|\int\int V_\infty(y,z)(e^{it\Delta}f_1e^{it\Delta}f_2)(x-y)(e^{it\Delta}f_3e^{it\Delta}f_4)(x-z)\,dy\,dz\|_{L^{3}_x}\\
&\hspace{2cm}\times\|\partial_j^{-1}P_je^{it\Delta}f_5\|_{L^6_x}\|\partial_jh\|_{L^2_x}\bigg\|_{L^1_t}\\
&\lesssim\sum_{j=1}^3\bigg\|\|V_\infty\|_{L^{\frac{1}{1-\epsilon}}}\|e^{it\Delta}f_1e^{it\Delta}f_2\|_{L^\frac{6}{1+12\epsilon}}\|e^{it\Delta}f_3e^{it\Delta}f_4\|_{L^6}\|\partial_j^{-1}P_je^{it\Delta}f_5\|_{L^6_x}\|\partial_jh\|_{L^2_x}\bigg\|_{L^1_t}\\
&\lesssim\sum_{j=1}^3\bigg\|\|V_\infty\|_{L^{\frac{1}{1-\epsilon}}}\prod_{\ell=1}^2\|e^{it\Delta}f_\ell\|_{L^\frac{12}{1+12\epsilon}_x}\prod_{m=3}^4 \|e^{it\Delta}f_m\|_{L^{12}_x} \|\partial_j^{-1}P_je^{it\Delta}f_5\|_{L^6_x}\|\partial_jh\|_{L^2_x}\bigg\|_{L^1_t}\\
&\lesssim\sum_{j=1}^3T^{3\epsilon}\|V_\infty\|_{L^{\frac{1}{1-\epsilon}}}\prod_{\ell=1}^2\|e^{it\Delta}f_\ell\|_{L^\frac{8}{1-12\epsilon}_t\dot W^{1,\frac{12}{5+12\epsilon}}_x}\prod_{m=3}^4\|e^{it\Delta}f_m\|_{L^8_t\dot W^{1,\frac{12}{5}}_x}\|\partial_j^{-1}P_je^{it\Delta}f_5\|_{L^2_tL^6_x}\|\partial_jh\|_{L^2_x}\\
&\lesssim T^{3\epsilon}\|V_\infty\|_{L^{\frac{1}{1-\epsilon}}}\|f_5\|_{\dot H^{-1}}\bigg(\prod_{\ell=1}^4\|f_\ell\|_{\dot H^1}\bigg)\|h\|_{\dot H^1}.\qedhere
\end{align*}
\end{proof}
\begin{proof}[Proof of \eqref{inequality:multilinear hartree H^{-1}} when $m\neq 5$]
We present the proof for $m=1$, and note that the proof for $m\in\{2,3,4\}$ is similar.  i.e. we show that
\begin{align*}
&\|A[V_\infty,(e^{it\Delta}f_1e^{it\Delta}f_2),(e^{it\Delta}f_3e^{it\Delta}f_4)]\cdot(e^{it\Delta}f_5)\|_{L^1_t\dot{H}^{-1}_x}\\
&\lesssim T^{3\epsilon}\|V_\infty\|_{L^\frac{1}{1-\epsilon}}\|f_1\|_{\dot{H}^{-1}}\prod_{\ell=2}^5\|f_\ell\|_{\dot{H}^1}.
\end{align*}
For $h\in \dot{H}^1(\mathbb{R}^3)$, we have
\begin{align*}
&\int A[V_\infty,(e^{it\Delta}f_1e^{it\Delta}f_2),(e^{it\Delta}f_3e^{it\Delta}f_4)](x)(e^{it\Delta}f_5)(x)\overline{h}(x)\,dx\\
&=\sum_{j=1}^3\int\int\int V_\infty(y,z)(\partial_j^{-1}P_je^{it\Delta}f_1)(x-y)\\
&\hspace{2cm}\times\partial_j\bigg[(e^{it\Delta}f_2)(x-y)(e^{it\Delta}f_3e^{it\Delta}f_4)(x-z)(e^{it\Delta}f_5)(x)\overline{h}(x)\bigg]\,dy\,dz\,dx\\
&=\sum_{j=1}^3\int\int\int V_\infty(y,z)(\partial_j^{-1}P_je^{it\Delta}f_1\cdot \partial_je^{it\Delta}f_2)(x-y)(e^{it\Delta}f_3e^{it\Delta}f_4)(x-z)\\
&\hspace{2cm}\times(e^{it\Delta}f_5)(x)\overline{h}(x)\,dy\,dz\,dx\\
&\hspace{1cm}+\text{four similar terms (by the product rule)}\\
&=:I_1+I_2+I_3+I_4+I_5.
\end{align*}
By duality, it now suffices to show that
\begin{align}
\|I_k\|_{L^1_t}\lesssim T^{3\epsilon}\|V_\infty\|_{L^{\frac{1}{1-\epsilon}}}\|f_1\|_{\dot H^{-1}}\bigg(\prod_{\ell=2}^5\|f_\ell\|_{\dot H^1}\bigg)\|h\|_{\dot H^1}\label{I_j}
\end{align}
holds for $k\in\{1,2,3,4,5\}$.  By Theorem \ref{beckner theorem}, Strichartz estimates, and Sobolev embedding, we have
\begin{align*}
\|I_1\|_{L^1_t}&\le\sum_{j=1}^3\bigg\|\|\int\int V_\infty(y,z)(\partial_j^{-1}P_je^{it\Delta}f_1\cdot\partial_je^{it\Delta}f_2)(x-y)(e^{it\Delta}f_3e^{it\Delta}f_4)(x-z)\,dy\,dz\|_{L^{\frac{3}{2}}_x}\\
&\hspace{2cm}\times\|e^{it\Delta}f_5\|_{L^6_x}\|h\|_{L^6_x}\bigg\|_{L^1_t}\\
&\lesssim\sum_{j=1}^3\bigg\|\|V_\infty\|_{L^{\frac{1}{1-\epsilon}}}\|\partial_j^{-1}P_je^{it\Delta}f_1\cdot\partial_je^{it\Delta}f_2\|_{L^\frac{3}{1+6\epsilon}}\|e^{it\Delta}f_3e^{it\Delta}f_4\|_{L^3}\|e^{it\Delta}f_5\|_{L^6_x}\|h\|_{L^6_x}\bigg\|_{L^1_t}\\
&\lesssim \sum_{j=1}^3 T^{3\epsilon}\|V_\infty\|_{L^{\frac{1}{1-\epsilon}}}\|\partial_j^{-1}P_je^{it\Delta}f_1\|_{L^{\frac{2}{1-3\epsilon}}_tL^\frac{6}{1+6\epsilon}_x}\|\partial_je^{it\Delta}f_2\|_{L^{\frac{2}{1-3\epsilon}}_tL^\frac{6}{1+6\epsilon}_x} \prod_{\ell=3}^5\|e^{it\Delta}f_\ell\|_{L^\infty_tL^6}\|h\|_{L^6}\\
&\lesssim T^{3\epsilon}\|V_\infty\|_{L^{\frac{1}{1-\epsilon}}}\|f_1\|_{\dot H^{-1}}\bigg(\prod_{\ell=2}^5\|f_\ell\|_{\dot H^1}\bigg)\|h\|_{\dot H^1},
\end{align*}
and similarly, $\eqref{I_j}$ holds for $k\in\{2,3,4\}$.  Finally, we bound $\|I_5\|_{L^1_t}$ by
\begin{align*}
&\sum_{j=1}^3\|\int\int V_\infty(y,z)(\partial_j^{-1}P_je^{it\Delta}f_1\cdot e^{it\Delta}f_2)(x-y)(e^{it\Delta}f_3e^{it\Delta}f_4)(x-z)\,dy\,dz\|_{L^1_tL^3_x}\\
&\hspace{2cm}\times\|e^{it\Delta}f_5\|_{L^\infty_tL^6_x}\|\partial_jh\|_{L^2_x}\\
&\le\sum_{j=1}^3\|V_\infty\|_{L^{\frac{1}{1-\epsilon}}}\|\partial_j^{-1}P_je^{it\Delta}f_1\cdot e^{it\Delta}f_2\|_{L^{\frac{3}{2}}_tL^{\frac{9}{2+18\epsilon}}_x}\|e^{it\Delta}f_3e^{it\Delta}f_4\|_{L^3_tL^9_x}\|e^{it\Delta}f_5\|_{L^\infty_tL^6_x}\|\partial_jh\|_{L^2_x}\\
&\le\sum_{j=1}^3\|V_\infty\|_{L^{\frac{1}{1-\epsilon}}}\|\partial_j^{-1}P_je^{it\Delta}f_1\|_{L^2_tL^\frac{6}{1+12\epsilon}_x}\prod_{\ell=2}^4\|e^{it\Delta}f_\ell\|_{L^6_tL^{18}_x} \|e^{it\Delta}f_5\|_{L^\infty_tL^6_x}\|\partial_jh\|_{L^2_x}\\
&\lesssim\sum_{j=1}^3T^{3\epsilon}\|V_\infty\|_{L^{\frac{1}{1-\epsilon}}}\|\partial_j^{-1}P_je^{it\Delta}f_1\|_{L^\frac{2}{1-6\epsilon}_tL^\frac{6}{1+12\epsilon}_x}\prod_{\ell=2}^4\|e^{it\Delta}f_2\|_{L^6_t\dot W^{1,\frac{18}{7}}_x}\|e^{it\Delta}f_5\|_{L^\infty_t\dot{H}^1_x}\|\partial_jh\|_{L^2_x}\\
&\lesssim T^{3\epsilon}\|V_\infty\|_{L^{\frac{1}{1-\epsilon}}}\|f_1\|_{\dot H^{-1}}\bigg(\prod_{\ell=2}^5\|f_\ell\|_{\dot H^1}\bigg)\|h\|_{\dot H^1}.\qedhere 
\end{align*}
\end{proof}

\bigskip
\footnotesize
\noindent\textit{Acknowledgments.}
The authors would like to express their special appreciation and thanks to their mentors Thomas Chen and Nata\v{s}a Pavlovi\'c for proposing the problem and for various useful discussions. Y.H. would like to thank IH\'ES for their hospitality and support while he visited in the summer of 2014. K.T. was supported by NSF grant and DMS-1151414 (CAREER, PI T. Chen)

\bibliographystyle{abbrv}

\end{document}

%% file: binary_tree_quintic.pdf_tex
\begingroup%
  \makeatletter%
  \providecommand\color[2][]{%
    \errmessage{(Inkscape) Color is used for the text in Inkscape, but the package 'color.sty' is not loaded}%
    \renewcommand\color[2][]{}%
  }%
  \providecommand\transparent[1]{%
    \errmessage{(Inkscape) Transparency is used (non-zero) for the text in Inkscape, but the package 'transparent.sty' is not loaded}%
    \renewcommand\transparent[1]{}%
  }%
  \providecommand\rotatebox[2]{#2}%
  \ifx\svgwidth\undefined%
    \setlength{\unitlength}{568.57678835bp}%
    \ifx\svgscale\undefined%
      \relax%
    \else%
      \setlength{\unitlength}{\unitlength * \real{\svgscale}}%
    \fi%
  \else%
    \setlength{\unitlength}{\svgwidth}%
  \fi%
  \global\let\svgwidth\undefined%
  \global\let\svgscale\undefined%
  \makeatother%
  \begin{picture}(1,0.80960795)%
    \put(0,0){\includegraphics[width=\unitlength]{binary_tree_quintic.pdf}}%
    \put(-1.28578699,1.09750378){\color[rgb]{0,0,0}\makebox(0,0)[lt]{\begin{minipage}{0.07236112\unitlength}\raggedright W1Wwwww\end{minipage}}}%
    \put(-0.00305373,0.78989957){\color[rgb]{0,0,0}\makebox(0,0)[lb]{\smash{$W_1$}}}%
    \put(0.02815656,0.69583699){\color[rgb]{0,0,0}\makebox(0,0)[lt]{\begin{minipage}{0.33039871\unitlength}\raggedright $B_{1;3,4}(v_1)$\end{minipage}}}%
    \put(0.13225933,0.81560629){\color[rgb]{0,0,0}\makebox(0,0)[lt]{\begin{minipage}{0.13582916\unitlength}\raggedright $W_2$\end{minipage}}}%
    \put(0.16265725,0.51742509){\color[rgb]{0,0,0}\makebox(0,0)[lb]{\smash{$B_{2;5,6}(v_2)$}}}%
    \put(0.00529768,0.014878){\color[rgb]{0,0,0}\makebox(0,0)[lb]{\smash{$u_1$}}}%
    \put(0.13144021,0.01202251){\color[rgb]{0,0,0}\makebox(0,0)[lb]{\smash{$u_2$}}}%
    \put(0.35143524,0.00657672){\color[rgb]{0,0,0}\makebox(0,0)[lb]{\smash{$u_4$}}}%
    \put(0.52916005,0.00649582){\color[rgb]{0,0,0}\makebox(0,0)[lb]{\smash{$u_6$}}}%
    \put(0.72615129,0.00643247){\color[rgb]{0,0,0}\makebox(0,0)[lb]{\smash{$u_8$}}}%
    \put(0.82368759,0.00399956){\color[rgb]{0,0,0}\makebox(0,0)[lb]{\smash{$u_9$}}}%
    \put(0.37606956,0.21197739){\color[rgb]{0,0,0}\makebox(0,0)[lb]{\smash{$B_{4;9,10}(v_4)$}}}%
    \put(0.39636975,0.44218467){\color[rgb]{0,0,0}\makebox(0,0)[lb]{\smash{$B_{4;7,8}(v_3)$}}}%
    \put(0.24978831,0.00795876){\color[rgb]{0,0,0}\makebox(0,0)[lb]{\smash{$u_3$}}}%
    \put(0.92856052,0.00595616){\color[rgb]{0,0,0}\makebox(0,0)[lb]{\smash{$u_{10}$}}}%
    \put(0.62791506,0.00660886){\color[rgb]{0,0,0}\makebox(0,0)[lb]{\smash{$u_7$}}}%
    \put(0.43668397,0.00629957){\color[rgb]{0,0,0}\makebox(0,0)[lb]{\smash{$u_5$}}}%
  \end{picture}%
\endgroup%